\documentclass[11pt]{amsart}
\usepackage{amsfonts}
\usepackage{amsmath}
\usepackage{txfonts}
\usepackage{amsthm}
\usepackage{mathrsfs}
\usepackage{graphicx}
\usepackage{latexsym,bm,amsmath,amssymb}
\usepackage{cite}
\usepackage{enumerate}
\usepackage{enumitem}
\usepackage{graphicx}
\usepackage{epstopdf}
\usepackage{epsfig}
\usepackage{overpic} 
\usepackage{calligra}
\usepackage[T1]{fontenc}
\usepackage{color}

\makeatletter
\@namedef{subjclassname@2020}{\textup{2020} Mathematics Subject Classification}
\makeatother

 \allowdisplaybreaks

\title[piecewise smooth generalized Abel equations with two asymmetric zones]{Bifurcation  of limit cycles in a class of piecewise smooth generalized Abel equations with two asymmetric zones}


\newtheorem{thm}{Theorem}[section]
\newtheorem{prop}[thm]{Proposition}
\newtheorem{lem}[thm]{Lemma}

\theoremstyle{remark}

\theoremstyle{definition}

\newtheorem{defn}[thm]{Definition}

\newtheorem{rem}[thm]{Remark}

\begin{document}
\author[ H.  Liang  and J. Huang]
{Haihua Liang$^1$,  Jianfeng Huang$^{2 *}$}

\address{$^1$
School of Mathematics and Systems Science, Guangdong Polytechnic Normal University,  Guangzhou, 510665, P.R.China}
\email{lianghhgdin@126.com}

\address{$^2$    Department of Mathematics,\ Jinan University,\
Guangzhou\ 510632,\ P.R.\ China}
\email{thuangjf@jnu.edu.cn}
\maketitle

\begin{abstract}
This paper studies the number of limit cycles, known as the Smale-Pugh problem, for the generalized Abel equation
\begin{align*}
  \frac{dx}{d\theta}=A(\theta)x^p+B(\theta)x^q,
\end{align*}
where $A$ and $B$ are are piecewise trigonometrical polynomials of degree $ m $ with two zones $0\leq\theta<\theta_1$ and $\theta_1\leq\theta\leq2\pi$.
By means of the first and second order analysis using the Melnikov theory and applying the new Chebyshev criterion that established by \cite{HLZ2023}, we estimate the maximum number of positive and negative limit cycles that such equations can have, and
reveal how this maximum number, denoted by $H_{\theta_1}(m)$, is affected by the location of the separation line $\theta=\theta_1$. For the equation of classical Abel type, our result not only includes the estimates provided in the recent paper (Huang et al., SIAM J. Appl. Dyn. Syst., 2020), i.e., $H_{2\pi}(m)\geq 4m-2$ for $\theta_1=2\pi$, but also shows that the equation in the discontinuous case can possess  more than two times as many limit cycles as in the continuous case. More accurately, $H_{\pi}(m)\geq 8m+2$  and  $H_{\theta_1}(m)\geq 14m-6$ for $\theta_1\in (0,\pi)\cup (\pi,2\pi)$.

\end{abstract}


\section{Introduction and statements of main results}

As a natural generalization of the Riccati equation,  the classical Abel  equation,
 \begin{equation*}
\frac{dx}{d\theta}=A(\theta)x^3+B(\theta)x^2+C(\theta)x+D(\theta)
\end{equation*}
is   a  very  useful tool for investigating the dynamics of  planar differential equations which produced from the real world.  A well known  example is that,  the Li\'{e}nard equation,
 which is a generalization of the equation
of motion of the damped oscillator,  can be transformed to the classical Abel  equation  with $C(x)\equiv D(x)\equiv 0$.
For more examples, the readers are refereed to \cite{HM2003}  for  the relativistic dissipative cosmological models, to \cite{TLM2014} for the susceptible-infected-recovered epidemic model, and so on.  The reduction of these models      to  the Abel  equation
     can greatly simplify the analysis of the dynamics.

With the increasing application of classical  Abel  equation,  the generalized Abel    equation
 \begin{equation*}
\frac{dx}{d\theta}=A_1(\theta)x^{i_1}+A_2(\theta)x^{i_2}+\cdots+A_n(\theta)x^{i_n},
\end{equation*}
where $i_1, i_2, \cdots, i_n\in \mathbf{Z}$, are   proposed  naturally. Since these kinds of equations   has a more wildly application in both the theory and the real worlds,  it has attracted  a lot of   interest of many researchers.
 An important class of the  generalized Abel  equations is the following one
\begin{equation}\label{specgeneabelequa}
\frac{dx}{d\theta}=A(\theta)x^p+B(\theta)x^q,
\end{equation}
where $A(\theta)$  and $B(\theta)$ are trigonometric polynomials, and $p$, $q$ are non-zero integers. Denote by $x(\theta,\rho)$ the solution of equation \eqref{specgeneabelequa} satisfying the initial condition  $x(0,\rho)=\rho$.
Then a solution $x(\theta,\rho)$ is periodic if $x(2\pi,\rho)=\rho$. A periodic solution is called to be a limit cycle if it is isolated in the set of periodic orbits. If $x=0$ is a solution of equation \eqref{specgeneabelequa}, then
 $x=0$ is called a center if there exists a neighborhood of it consisting of only periodic solutions.

The most concerned problems on   equation \eqref{specgeneabelequa}
 are the existence of center and the number of  limit cycles.
To find  the  sufficient
and necessary conditions such that  $x=0$  is a center is    the so-called center problem. This problem has absorbed  much interests of researchers because it is  not only
closely related to the famous center-focus problem for the planar polynomial differential
systems, but also  of great theoretic importance itself. Some interesting results  can be found in the recent works  \cite{Zhou2020} by Zhou,  \cite{LV2020} by  J. Llibre and C. Valls, and \cite{LWang2023} by Liu and Wang,
where \cite{LWang2023} prove  that   $x=0$  is a center of \eqref{specgeneabelequa} whenever   $A, B$ are odd functions and $p>q>1$, without the restriction $p\geq 2q-1$  in  \cite{LV2020}.

Determining the number of limit cycles    for equation \eqref{specgeneabelequa}   is another   important  problem which   can be applied to the automatic control, biology, physical mode and so on.
For  example, in     \cite{GGM2016,GGM2020},   the authors   study the perturbed pendulum systems and the perturbed whirling pendulum systems of the form \begin{align*}
\left\{
\begin{aligned}
&\dot{x}=y    ,  \\
&\dot{y}=-\sin x+\varepsilon Q(x)y^{p} ,
\end{aligned}
\right.\
\text{  and  }\
\left\{
\begin{aligned}
&\dot{x}=y ,    \\
&\dot{y}=\sin x\left(\cos x-\gamma\right)+\varepsilon Q(x)y^{p},
\end{aligned}
\right.
\end{align*}
respectively, here  $ Q$ is trigonometrical polynomial of degree $m$. By changing the systems to the form of \eqref{specgeneabelequa}
\begin{equation*}
\frac{dy}{dx}=A(x)y^{-1}+\varepsilon Q(x)y^{p-1},\indent A(x)=-\sin x \text{ or } \sin x\left(\cos x-\gamma\right),
\end{equation*}
they obtain   the upper bounds for the number of limit cycles bifurcation from the rotary region.

It is well known that, the equation \eqref{specgeneabelequa} can has arbitrary many limit cycles if  no additional conditions are imposed.  Indeed,    early in 1980, by using the perturbation approach,  Lins-Neto \cite{Neto1980}  shows that the  equation \eqref{specgeneabelequa} with $(p,q)=(3,2)$
 can have at least $m$ limit cycles, where $m$ is the maximal degree of the trigonometrical polynomials $A(\theta)$
and $B(\theta)$. This is the first result showing that the  maximum number of limit cycles of  equation \eqref{specgeneabelequa} is not bounded.
For this reason, a new problem arise naturally as follows:  Let $n$ and $m$ be the degrees of the  trigonometric polynomials $A(\theta) $  and $B(\theta) $ respectively. For fixed  integers $n$, $m$ and $p$, $q$, define the Hilbert number
$\mathcal H_{p,q}(n,m)$ as the maximum number  of limit cycles that equation \eqref{specgeneabelequa} can have. Then, find out the value of  $\mathcal H_{p,q}(n,m)$ (just like the Hilbert' 16th problem for the planar polynomial  differential systems).  This is   also a very difficult  problem  that
 until now it is even unknown whether  $\mathcal H_{p,q}(n,m)<+\infty$.
 This problem in the case $(p,q)=(3,2)$ is the so-called Smale-Pugh problem \cite{Smale1998}.

In \cite{HTC2020}, in order to improve the lower bounds for $\mathcal H_{p,q}(n,m)$, the authors study a second order perturbation of the generalized Abel equation
\begin{equation}\label{equa1}
\frac{dx}{d\theta}=(\sin \theta+\varepsilon P_1(\theta)+\varepsilon^2 P_2(\theta))x^{p}+(\varepsilon Q_1(\theta)+\varepsilon^2 Q_2(\theta))x^{q},
\end{equation}
where for each $i=1,2$, $P_i(\theta)$ and $Q_i(\theta)$ are respectively the trigonometric polynomials of degree $n$ and $m$. The authors provide a new lower bound for the case that $(p,q)=(3,2)$, that is,  $\mathcal H_{3,2}(n,m)\geq 2(n+m)-1$.
This  is an important improvement of the  known results for  Smale-Pugh problem.

Very recently, promoted by the fact that there are a lot of  sudden behaviours and discontinuous phenomena occurring in the real world,   more and more  mathematicians  began to
investigate the dynamics  of piecewise smooth differential systems, including the piecewise smooth planar systems  \cite{BBCK2008}, \cite{EP2023}, \cite{LL2023},  \cite{LV2023}, \cite{WZ2016} and the piecewise smooth   Abel equations \cite{HL2022},
\cite{HP2020}, \cite{HLZ2023},  \cite{ZhYuW2022}, where the latter
can   sometimes turn out to be a  useful tool for studying the  former.

Inspired by the work of \cite{HLZ2023}, \cite{HTC2020},  in this paper  we consider equation \eqref{specgeneabelequa} for the piecewise smooth  case.
In \cite{HLZ2023},  the authors establish a new chebyshev criterion which can be applied to study  the bifurcation problem of piecewise smooth systems. And we recall that  the authors in  \cite{HTC2020}  study  the numbers of limit cycles of  equation \eqref{specgeneabelequa}  by analyzing  the  second order Melnikov functions for equation    \eqref{equa1}.  In particular,  if we set $m=\mbox{max}\{\mbox{deg}A, \mbox{deg}B\}$, then  they actually  show that
   $H_{3,2}(m)\geq 4m-1$, where one limit cycle is the zero solution.
 Thus, when we go further and consider equation \eqref{specgeneabelequa} for  the  piecewise smooth case,  some interesting    problems    arises naturally as follows.

  {\bf Problems for equation   \eqref{specgeneabelequa}}. If both $A(\theta)$ and $B(\theta)$ are  the piecewise trigonometric polynomials with a separatrix $\theta=\theta_1$,  then what is the maximum number of limit cycles that  equation \eqref{specgeneabelequa} can has?
Is this number evidently larger then the one in the continuous case?  Moreover, how does the position of $\theta_1$ affect this number?

 With the goal to answer these questions, we focus on   equation \eqref{equa1} for the piecewise smooth case. More   precisely, we assume for \eqref{equa1} that,
 $  p,q\in \mathbf{Z}\backslash\{0, 1\}$, $ \frac{q-1}{p-1}\notin \mathbf{Z}_{\leq1} $, and $P_i$, $ Q_i$ are piecewise trigonometrical polynomials of degree $ m $ with two zones $0\leq\theta<\theta_1$ and $\theta_1\leq\theta\leq2\pi$. That is, for $i=1,2$,
\begin{align}\label{equa1-1}
&P_i(\theta):=
\left\{
\begin{aligned}
P_i^+(\theta)=\sum_{k=0}^{m}\big(a_{ik}^{+}\sin (k\theta) + b_{ik}^{+}\cos (k\theta)\big)    , && 0   \le  \theta <  \theta_{1},  \\
P_i^-(\theta)=\sum_{k=0}^{m}\big(a_{ik}^{-}\sin (k\theta) + b_{ik}^{-}\cos (k\theta)\big)    , && \theta_{1} \leq \theta \le 2\pi,
\end{aligned}
\right.
\end{align}

\begin{align}\label{equa1-2}
&Q_i(\theta):=
\left\{
\begin{aligned}
Q_i^+(\theta)=\sum_{k=0}^{m}\big(c_{ik}^{+}\sin (k\theta) + d_{ik}^{+}\cos (k\theta)\big)    , && 0   \le  \theta <  \theta_{1},  \\
Q_i^-(\theta)=\sum_{k=0}^{m}\big(c_{ik}^{-}\sin (k\theta) + d_{ik}^{-}\cos (k\theta)\big)    , && \theta_{1} \leq \theta \le 2\pi.
\end{aligned}
\right.
\end{align}
For $\theta_1=2\pi$, we impose the parameter restrictions that $a_{ik}^+=a_{ik}^-, b_{ik}^+=b_{ik}^-$ and $c_{ik}^+=c_{ik}^-, d_{ik}^+=d_{ik}^-$ because the equation is actually smooth in this case.

The solution  $x_{\varepsilon}(\theta,\rho)$ of equation \eqref{equa1},  with the initial value $\rho$, is called to be periodic, if it is well-defined on $[0, 2\pi]$ and satisfies $x_{\varepsilon}(2\pi,\rho)=\rho$.
  By direct computation one can get that
\begin{equation}\label{equa1-3}
 x_{0}^{1-p}(\theta,\rho)=(1-p)(1-\cos\theta)+\rho^{1-p}.
\end{equation}
Thus, the periodic annulus of equation \eqref{equa1}$_{\varepsilon=0}$ is given by  $$U:=\{(\theta, x): 0\leq\theta\leq 2\pi, x^{1-p}+(p-1)(1-\cos\theta)-\rho^{1-p}=0, \rho\in I\}$$ with

\begin{equation*}\label{equa1I}
I:=
\left\{
\begin{aligned}
(-h(p),0)\cup (0, h(p)), &&  \mbox{if} \ p>1 \ \mbox{is odd},  \\
(-\infty,0)\cup (0, h(p)), &&  \mbox{if}\  p>1 \ \mbox{is even},\\
(-\infty,0)\cup (0, +\infty), &&  \mbox{if}\  p<1 \ \mbox{is odd},\\
(-\infty,h(p))\cup (0, +\infty), &&  \mbox{if}\  p<1 \ \mbox{is even},
\end{aligned}
\right.
\end{equation*}
where $h(p)=(2p-2)^{1/(1-p)}$.

We remark that from the properties of composite functions,  $x_{\varepsilon}(\theta,\rho)$ is smooth in  $\varepsilon,\rho$, and piecewise smooth in  $\theta$, respectively.
Note that for any $\rho\in I$, $x_{\varepsilon}(\theta,\rho) \neq0$ when  $\varepsilon$ is small enough.  We take the change of variable
\begin{equation}\label{equa2}
y=\frac{1}{1-p}x^{1-p}.
\end{equation}
Then, the   equation \eqref{equa1} is reduced to
\begin{equation}\label{equa3}
\frac{dy}{d\theta}
=\sin \theta+\varepsilon P_1(\theta)+\varepsilon^2 P_2(\theta) +(\varepsilon \tilde{Q}_1(\theta)+\varepsilon^2 \tilde{Q}_2(\theta)) y^{\alpha},
\end{equation}
where $\alpha=\frac{q-p}{1-p}$,   $\tilde{Q}_i(\theta)=(1-p)^{\alpha}Q_i(\theta)$,  $i=1,2$.   Denote by $ y_{\varepsilon}(\theta,\hat{\rho})$ the solution of \eqref{equa3} with initial condition $ y_{\varepsilon}(0,\hat{\rho})=\hat{\rho} $, where $ \hat{\rho}=\frac{1}{1-p}\rho^{1-p} $. One can readily verify that $y_{\varepsilon}(\theta,\hat{\rho})=\frac{1}{1-p}x_{\varepsilon}(\theta,\rho)^{1-p} $ with $y_0(\theta,\hat{\rho})=\hat{\rho}+1-\cos \theta$.

Due to \eqref{equa1-3} and \eqref{equa2},  the periodic annulus of equation \eqref{equa3}$_{\varepsilon=0}$ becomes  $$\hat{U}:=\{(\theta, y): 0\leq\theta\leq 2\pi, y=\hat{\rho}+1-\cos \theta, \hat{\rho}\in \hat{I}\}$$ where
\begin{equation}\label{equa1-4}
\hat{I}:=
\left\{
\begin{aligned}
(-\infty,-2),\quad\quad\quad\quad             &&  \mbox{if} \ p>1 \ \mbox{is odd},  \\
(0, +\infty),   \quad\quad\quad\quad\               &&  \mbox{if}\  p<1 \ \mbox{is odd},\\
(-\infty,-2)\cup (0, +\infty), &&  \mbox{if}\  p\neq 1 \ \mbox{is even}.
\end{aligned}
\right.
\end{equation}

Clearly, $ y_{\varepsilon}(\theta,\hat{\rho}) $ can be expanded into the Taylor series at $ \varepsilon=0 $:
\begin{equation*}\label{equa1-4-1}
y_{\varepsilon}(\theta,\hat{\rho})=y_{0}(\theta,\hat{\rho})+\sum_{i=1}^{\infty}\hat{S}_{i}(\theta,\hat{\rho})\varepsilon^{i}, \quad \text{where $ \hat{S}_{i}(\theta,\hat{\rho})=\frac{1}{i!}\partial_{\varepsilon}^{i}y_{\varepsilon}(\theta,\hat{\rho})|_{\varepsilon=0} $},\
\hat{S}_{i}(0,\hat{\rho})=0.
\end{equation*}
For simplicity in what follows we let $\hat{\rho}\to\rho$, and therefore
\begin{equation}\label{equa1-5}
y_{\varepsilon}(2\pi,\rho)-\rho=\sum_{i=1}^{\infty}\varepsilon^{i}M_{i}(\rho)   \quad \text{with} \quad M_{i}(\rho)=\frac{1}{i!}\partial_{\varepsilon}^{i}\left(y_{\varepsilon}(2\pi,\rho)\right)|_{\varepsilon=0}=\hat{S}_{i}(2\pi,\rho).
\end{equation}
For each $i\in\mathbb Z^+$, $ M_{i}(\rho) $ is called the $i$-th order \emph{Melnikov function} of equation \eqref{equa3}.
We remark that the number of limit cycles bifurcating form the periodic annulus is determined by the number of isolated zeros of    $M_n$  provided $M_1(\rho)=\cdots=M_{n-1}(\rho)\equiv 0$ and $M_n(\rho)\not\equiv 0$.

The main results of this paper are the following theorems.

\begin{thm} \label{thm1-1} Assume that $M_n$ is the first non-vanishing Melnikov function of equation \eqref{equa3}. Let $Z_n(m)$  ($m\geq1$) be the maximum number of isolated zeros of $M_n$  on $\hat{I}$,
counted with multiplicity. Then, the value of $Z_n(m)$  with respect to the value of $\theta_1$   are
given in Table \ref{table1}.

\end{thm}
\begin{table}[hpt]\label{table1}
 \centering
 \renewcommand\arraystretch{1.5}
\begin{tabular}{|l|*{6}{c|}}\hline
 \quad &$\theta_{1}\in(0,\pi)\cup(\pi,2\pi)$ &$\theta_{1}=\pi$  &$\theta_{1}=2\pi$  \\ \hline
 $Z_1(m)$   & $3m+2$        & $2m+1$   & $m+1$  \\  \hline
 $Z_2(m)$: & $\in [7m-3, 9m-4]$        & $\geq 4m+1$   &$\geq 2m-1$  \\ \hline
\end{tabular}
\vskip0.1cm
\caption{The values of $Z_1(m)$ and $Z_2(m)$. }\vspace{-0.7cm}
\end{table}

Theorem \ref{thm1-1} actually give the lower bound of  the maximum  numbers of limit cycles of equation \eqref{equa3} for sufficient small $\varepsilon$.   Due to the transformation \eqref{equa2},   the number of non-zero limit cycles of equation \eqref{equa1} is twice (resp. the same) as the one of equation \eqref{equa3} when $p$ is odd (resp. even).  Therefore,
from Theorem \ref{thm1-1}, we  can immediately obtain the next conclusion.

\begin{thm}\label{Theorem1-3}
The maximum number $\mathcal H_{\theta_1}(m)$   ($m\geq1$) of positive and negative limit cycles for equation \eqref{specgeneabelequa}, with respect to the value of $\theta_1$ and the parities of $p$, verifies the estimates in Table \ref{table2}.
\begin{table}[hpt]\label{table2}
 \centering
 \renewcommand\arraystretch{1.5}
\begin{tabular}{|l|*{6}{c|}}\hline
 &$\theta_{1}\in(0,\pi)\cup(\pi,2\pi)$ &$\theta_{1}=\pi$ &$\theta_{1}=2\pi$  \\ \hline
   $p$ is even&$\geq 7m-3$ &   $\geq 4m+1$   & $\geq 2m-1$  \\ \hline
   $p$ is odd &$\geq 14m-6$ &$\geq 8m+2$  &$\geq 4m-2$   \\ \hline
\end{tabular}
\vskip0.1cm
\caption{Estimates of $\mathcal H_{\theta_1}(m)$. }\vspace{-0.7cm}
\end{table}
\end{thm}

\begin{rem} (1)   Theorem \ref{Theorem1-3}  includes the result of \cite{HTC2020} which shows   that when $(p,q)=(3,2)$ the lower bound of numbers of limit cycles of smooth  Abel equation \eqref{equa1} is not less than $4m-1$. Indeed,
   when $\theta_{1}=2\pi$, the piecewise smooth  Abel equation \eqref{equa1} has
the same dynamical    characteristics as that of  the smooth equation. From Theorem \ref{Theorem1-3}, we know that $\mathcal H_{2\pi}(m)\geq 4m-2$ provided  $p$ is odd. Taking into account that when $p, q>0$ we have another limit cycle $x(\theta)\equiv 0$, thus
           we can produce at least $4m-1$ limit cycles for small $\varepsilon$.

           (2) By Theorem \ref{Theorem1-3}, for the trigonometric  polynomials Abel equation \eqref{equa1},  the  lower bound of the number of limit cycles  for  the piecewise smooth case,
is obviously larger than, the smooth cases. Moveover, the position of   $\theta_1$ has two critical values,  namely  $\pi$ and $2\pi$,   which affect obviously the numbers  of limit cycles of   equation \eqref{equa1}.
\end{rem}

The rest of this paper is organized as follows. In section 2 we mainly study the structure of the first and second order  Melnikov functions of equation \eqref{equa3}, while   section 3 is devoted to  establishing the  Chebyshev family
for the first non-vanishing Melnikov functions of equation \eqref{equa3}.  In section 4, by using the  Chebyshev family provided in section 3, we prove Theorem  \ref{thm1-1}.

\section{Melnikov function for equation \eqref{equa3}}
The aim of this  section is to  deduce  the expression of the first  non-vanishing Melnikov function for  equation \eqref{equa3}, namely $M_n(\rho)$,  for $n=1,2$.

 From   \eqref{equa1-5}  we obtain that
  $$ M_{i}(\rho)=\hat{S}_{i}(2\pi,\rho)=\hat{S}_{i}(0,\rho)+\int_{0}^{2\pi}\partial_{\theta} \hat{S}_{i}(\theta,\rho)d \theta=\int_{0}^{2\pi}\partial_{\theta} \hat{S}_{i}(\theta,\rho)d \theta.$$

Furthermore, for $n=1,2$,
\begin{align}\label{equa4}
\begin{split}
n!\partial_{\theta}\hat{S}_{n}(\theta,\rho)&=\partial_{\varepsilon}^{n}\Big(\partial_{\theta}y_{\varepsilon}(\theta,\rho)\Big)\Big|_{\varepsilon=0}\\
&=\partial_{\varepsilon}^{n}\left(\sin \theta+\varepsilon P_1(\theta)+\varepsilon^2 P_2(\theta) +(\varepsilon \tilde{Q}_1(\theta)+\varepsilon^2 \tilde{Q}_2(\theta))\cdot y_{\varepsilon}^{\alpha}(\theta,\rho)\right)\bigg|_{\varepsilon=0}\\
&=nP_n(\theta)+\sum_{i=0}^{n}\dbinom{n}{i}\partial_{\varepsilon}^{i}(\varepsilon \tilde{Q}_1(\theta)+\varepsilon^2 \tilde{Q}_2(\theta))\partial_{\varepsilon}^{n-i} y_{\varepsilon}^{\alpha}(\theta,\rho) \Bigg|_{\varepsilon=0}.
\end{split}
\end{align}
where $\alpha=(q-p)/(1-p)$.

   Thus,

\begin{equation}\label{equa6}
M_{1}(\rho)=p_{10}+\int_{0}^{2\pi}\tilde{Q}_1(\theta)y_0(\theta,\rho)^{\alpha}d \theta,
\end{equation}
and
\begin{equation}\label{equa7}
M_{2}(\rho)=p_{20}+\int_{0}^{2\pi} \big(\tilde{Q}_2(\theta)y_0(\theta,\rho)^{\alpha}+\alpha\tilde{Q}_1(\theta)y_0(\theta,\rho)^{\alpha-1}\hat{S}_1(\theta,\rho)\big)d \theta,
\end{equation}
 where
\begin{equation*}
\label{equa7-01}p_{10}=\int_{0}^{2\pi}P_1(\theta)d \theta, \quad p_{20}=\int_{0}^{2\pi}P_2(\theta)d \theta.\end{equation*}

The next step  is to write $M_{1}(\rho)$ and $M_2(\rho)$ as   linear combinations  of the following functions
\begin{equation}\label{equa2-1}
\begin{split}
&\mathcal{C}_{k}^{E}(\rho,\beta):=\int_{E}\cos (k\theta)\left(1+\rho-\cos\theta\right)^{\beta}d\theta, \\ &\mathcal{S}_{k}^{E}(\rho,\beta):=\int_{E}\sin (k\theta)\left(1+\rho-\cos\theta\right)^{\beta}d\theta,
\\ &\mathcal{D}_{k}^{E}(\rho,\beta):=\int_{E}\theta\sin (k\theta)\left(1+\rho-\cos\theta\right)^{\beta}d\theta,
\end{split}
\end{equation}
where  $E\subseteq [0,2\pi]$, $\rho \in \hat{I}$ and  $\beta\notin \mathbf{Z}_{\geq 0}$.  And, to ensure the  conciseness  of these expressions as possible, we will utilize the following properties
 of $\mathcal{C}_{k}^{E}, \mathcal{S}_{k}^{E}$ and $\mathcal{D}_{k}^{E}$.

\begin{lem} \label{lemforSCD} Let  $\mathcal{C}_{k}^{E}, \mathcal{S}_{k}^{E}$ and $\mathcal{D}_{k}^{E}$ be the functions defined by  \eqref{equa2-1}.
\begin{itemize}

\item[(i)] If  $\theta_1\in(0,\pi]$, then \begin{equation*}
\label{relaofscd-1}\mathcal{S}_{k}^{\left[\theta_{1},2\pi\right]}
=-\mathcal{S}_{k}^{\left[0,\theta_{1}\right]},\
\mathcal{C}_{k}^{\left[\theta_{1},2\pi\right]}
=\mathcal{C}_{k}^{\left[0,\theta_{1}\right]}+2\mathcal{C}_{k}^{\left[\theta_{1},\pi\right]},\
\mathcal{D}_{k}^{\left[\theta_{1},2\pi\right]}
=\mathcal{D}_{k}^{\left[0,\theta_{1}\right]}+2\mathcal{D}_{k}^{\left[\theta_{1},\pi\right]}-2\pi\left(\mathcal{S}_{k}^{\left[0,\theta_{1}\right]}+\mathcal{S}_{k}^{\left[\theta_{1},\pi\right]}\right).\end{equation*}

\item[(ii)]  If $\theta_1\in(\pi, 2\pi]$, then
\begin{equation*}\label{relaofscd-2}
 \mathcal{S}_{k}^{\left[0,\theta_{1}\right]}=\mathcal{S}_{k}^{\left[0,2\pi-\theta_{1}\right]}=-\mathcal{S}_{k}^{\left[\theta_{1},2\pi\right]},\quad
\mathcal{C}_{k}^{\left[0,\theta_{1}\right]}
=\mathcal{C}_{k}^{\left[0,2\pi-\theta_{1}\right]}+2\mathcal{C}_{k}^{\left[2\pi-\theta_{1},\pi\right]},\quad
\mathcal{C}_{k}^{\left[\theta_{1},2\pi\right]}
=\mathcal{C}_{k}^{\left[0,2\pi-\theta_1\right]},
\end{equation*}
\begin{equation*}\label{relaofscd-2-1}
 \mathcal{D}_{k}^{\left[\theta_{1},2\pi\right]}=\mathcal{D}_{k}^{\left[0, 2\pi-\theta_{1}\right]}-2\pi\mathcal{S}_{k}^{\left[0, 2\pi-\theta_{1}\right]},\quad
 \mathcal{D}_{k}^{\left[0, \theta_{1}\right]}=\mathcal{D}_{k}^{\left[0, 2\pi-\theta_{1}\right]}+2\mathcal{D}_{k}^{\left[ 2\pi-\theta_{1}, \pi\right]}-2\pi\mathcal{S}_{k}^{\left[ 2\pi-\theta_{1}, \pi\right]}.
\end{equation*}
\end{itemize}
\end{lem}

\begin{proof} (i)
We only provide the proof of the third equality since the first and the second equalities can be dealt with in a similar way. By the transformation $\theta \to 2\pi-\theta$, we get
\begin{equation*} \label{relaprooeq1}
\begin{split}
\mathcal{D}_{k}^{\left[\theta_{1},2\pi-\theta_1\right]}=&\mathcal{D}_{k}^{\left[\theta_{1},\pi\right]}+\mathcal{D}_{k}^{\left[\pi,2\pi-\theta_1\right]}\\
=&\mathcal{D}_{k}^{\left[\theta_{1},\pi\right]}+\int_{\theta_1}^{\pi}(\theta-2\pi)\sin(k\theta)\left(1+\rho-\cos\theta\right)^{\beta}d\theta\\
=&2\mathcal{D}_{k}^{\left[\theta_{1},\pi\right]}-2\pi\mathcal{S}_{k}^{\left[\theta_{1},\pi\right]},
\end{split}
\end{equation*}
and
\begin{equation*}
\mathcal{D}_{k}^{\left[2\pi-\theta_{1},2\pi\right]}=\int_0^{\theta_1}(\theta-2\pi)\sin(k\theta)\left(1+\rho-\cos\theta\right)^{\beta}d\theta=\mathcal{D}_{k}^{\left[0,\theta_{1}\right]}-2\pi\mathcal{S}_{k}^{\left[0,\theta_{1}\right]}.
\end{equation*}
Thus, $$\mathcal{D}_{k}^{\left[\theta_{1},2\pi\right]}=\mathcal{D}_{k}^{\left[\theta_{1},2\pi-\theta_1\right]}+\mathcal{D}_{k}^{\left[2\pi-\theta_{1},2\pi\right]}
=\mathcal{D}_{k}^{\left[0,\theta_{1}\right]}+2\mathcal{D}_{k}^{\left[\theta_{1},\pi\right]}-2\pi\mathcal{S}_{k}^{\left[0,\theta_{1}\right]}-2\pi\mathcal{S}_{k}^{\left[\theta_{1},\pi\right]}.$$

 (ii)
We only provide the proof of the   fourth and fifth  equalities. Using  the transformation $\theta \to 2\pi-\theta$, we get
\begin{equation*} \label{relaprooeq2}
\mathcal{D}_{k}^{\left[0,\theta_{1}\right]}=\int_{2\pi-\theta_1}^{2\pi}(\theta-2\pi)\sin(k\theta)\left(1+\rho-\cos\theta\right)^{\beta}d\theta=\mathcal{D}_{k}^{\left[0,2\pi\right]}-\mathcal{D}_{k}^{\left[0,2\pi-\theta_{1}\right]}+2\pi\mathcal{S}_{k}^{\left[0,2\pi-\theta_{1}\right]},
\end{equation*}
thus \begin{equation*}
\mathcal{D}_{k}^{\left[\theta_{1}, 2\pi\right]}=\mathcal{D}_{k}^{\left[0,2\pi\right]}-\mathcal{D}_{k}^{\left[0,\theta_{1}\right]}=\mathcal{D}_{k}^{\left[0,2\pi-\theta_{1}\right]}-2\pi\mathcal{S}_{k}^{\left[0,2\pi-\theta_{1}\right]}.
\end{equation*}

Next we prove the fifth equality. Again, by using the transformation $\theta \to 2\pi-\theta$,
\begin{equation*}
\mathcal{D}_{k}^{\left[\pi,\theta_{1}\right]}=\int_{2\pi-\theta_1}^{\pi}(\theta-2\pi)\sin(k\theta)\left(1+\rho-\cos\theta\right)^{\beta}d\theta=\mathcal{D}_{k}^{\left[2\pi-\theta_{1},\pi\right]}-2\pi\mathcal{S}_{k}^{\left[2\pi-\theta_{1},\pi\right]}.
\end{equation*}
Hence,
 \begin{equation*}
\mathcal{D}_{k}^{\left[0,\theta_{1}\right]}=\mathcal{D}_{k}^{\left[0,2\pi-\theta_1\right]}+\mathcal{D}_{k}^{\left[2\pi-\theta_1,\pi\right]}+\mathcal{D}_{k}^{\left[\pi,\theta_{1}\right]}=\mathcal{D}_{k}^{\left[0,2\pi-\theta_1\right]}+2\mathcal{D}_{k}^{\left[2\pi-\theta_{1},\pi\right]}-2\pi\mathcal{S}_{k}^{\left[2\pi-\theta_{1},\pi\right]}.
\end{equation*}

The proof is complete. \end{proof}

\subsection{Expression of $M_{1}(\rho)$.}

From \eqref{equa6} and  Lemma \ref{lemforSCD}, we can obtain the  expression of  $M_{1}(\rho)$ by direct calculation.

\begin{prop} \label{prop1}Denoted by  $\tilde{c}_{ij}^{\pm}=(1-p)^{\alpha}c_{ij}^{\pm}$ and  $\tilde{d}_{ij}^{\pm}=(1-p)^{\alpha}d_{ij}^{\pm}$.
The following statements are true for \eqref{equa3}.
\begin{itemize}

\item[(i)] If $\theta_1\in(0,\pi)$,  then
\begin{equation*}\label{equa7-1}
M_{1}(\rho)=p_{10}+\sum_{k=0}^{m}\left(\tilde{d}_{1k}^{+}+\tilde{d}_{1k}^{-}\right)\mathcal{C}_{k}^{E_1}(\rho,\alpha)
+\sum_{k=1}^{m}\left(\tilde{c}_{1k}^{+}-\tilde{c}_{1k}^{-}\right)\mathcal{S}_{k}^{E_1}(\rho,\alpha)
+2\sum_{k=0}^{m}\tilde{d}_{1k}^-\mathcal{C}_{k}^{E_2}(\rho,\alpha),
\end{equation*}
where $E_1=[0,\theta_1], E_2=[\theta_1, \pi]$.

\item[(ii)] If   $\theta_1\in(\pi, 2\pi)$,  then
\begin{equation*}\label{equa8}
M_{1}(\rho)=p_{10}+\sum_{k=0}^{m}\left(\tilde{d}_{1k}^{+}+\tilde{d}_{1k}^{-}\right)\mathcal{C}_{k}^{E_1}(\rho,\alpha)
+\sum_{k=1}^{m}\left(\tilde{c}_{1k}^{+}-\tilde{c}_{1k}^{-}\right)\mathcal{S}_{k}^{E_1}(\rho,\alpha)
+2\sum_{k=0}^{m}\tilde{d}_{1k}^+\mathcal{C}_{k}^{E_2}(\rho,\alpha),
\end{equation*}
where $E_1=[0,2\pi-\theta_1], E_2=[2\pi-\theta_1, \pi]$.

\item[(iii)]  If  $\theta_1\in \{\pi, 2\pi\}$, then
 \begin{align*}\label{equa1-40}
M_{1}(\rho)=
\left\{
\begin{aligned}
  p_{10}+\sum_{k=0}^{m}\left(\tilde{d}_{1k}^{+}+\tilde{d}_{1k}^{-}\right)\mathcal{C}_{k}^{E}(\rho,\alpha)
+\sum_{k=1}^{m}\left(\tilde{c}_{1k}^{+}-\tilde{c}_{1k}^{-}\right)\mathcal{S}_{k}^{E}(\rho,\alpha),&& \mbox{for} \ \theta_1=\pi,  \\
 p_{10}+2\sum_{k=0}^{m}\tilde{d}_{1k}^{+}\mathcal{C}_{k}^{E}(\rho,\alpha),\quad \quad\quad\quad\quad\quad\quad\quad\quad\quad
  \quad\quad&&  \mbox{for} \ \theta_1=2\pi,
\end{aligned}
\right.
\end{align*}
where $E=[0, \pi]$.
\end{itemize}

\end{prop}

\subsection{Expression of $M_{2}(\rho)$.}
In this subsection we shall     study the  second order   Melnikov function $M_2(\rho)$, under the assumption    that $M_1(\rho)\equiv 0$. To this end, we need to simplify the
expression of $\tilde{Q}_1$.

 \begin{prop} \label{prop1+1} Suppose that $M_1(\rho)\equiv 0$. Then $p_{10}=0$ and the following statements are true.
 \begin{itemize}
 \item[(i)] If $\theta_1\in (0, \pi)\cup(\pi, 2\pi]$,  then $\tilde{d}_{1k}^{+}=\tilde{d}_{1k}^{-}=0$ and $ \tilde{c}_{1k}^{+}=\tilde{c}_{1k}^{-}$.
  If   $\theta_1=\pi$,  then   $\tilde{d}_{1k}^{+}=-\tilde{d}_{1k}^{-}$ and $ \tilde{c}_{1k}^{+}=\tilde{c}_{1k}^{-}$.

\item[(ii)] $ \tilde{Q}_1(\theta+\pi)$ is an odd function on $[-\pi, \pi]$.
\end{itemize}
\end{prop}
\begin{proof} (i) First assume  that  $\theta_1\in (0, \pi)\cup(\pi, 2\pi)$.
For any  $ \vartheta\in (0, \pi)$, let  $E_1=[0, \vartheta], E_2=[\vartheta, \pi]$.  From  Proposition \ref{prop3-5} (see Section 3), the functions
\begin{equation*}\mathcal{C}_{0}^{E_1}(\rho,\alpha), \mathcal{C}_{1}^{E_1}(\rho,\alpha), \cdots, \mathcal{C}_{m}^{E_1}(\rho,\alpha), \mathcal{S}_{m}^{E_1}(\rho,\alpha), \cdots, \mathcal{S}_{1}^{E_1}(\rho,\alpha),
 \mathcal{C}_{0}^{E_2}(\rho,\alpha), \mathcal{C}_{1}^{E_2}(\rho,\alpha), \cdots, \mathcal{C}_{m}^{E_2}(\rho,\alpha)\end{equation*} are independent.  Therefore, we get from Proposition \ref{prop1}   that
 $p_{10}=0$ and $\tilde{d}_{1k}^{+}=\tilde{d}_{1k}^{-}=0, \tilde{c}_{1k}^{+}=\tilde{c}_{1k}^{-}$.  The case $\theta_1\in \{\pi, 2\pi\}$ can be treated similarly.

 (ii) The conclusion follows from the \eqref{equa1-2} and result of (i) directly.
\end{proof}

Under the assumption  $M_1(\rho)\equiv 0$, we know that
 $L_1(\theta):= \int_{0}^{\theta+\pi} \tilde{Q}_1(s)y_0(s,\rho)^{\alpha}d s$ is even function on $[-\pi, \pi]$, due to the fact that
$$L_1(\theta)-L_1(-\theta)= \int_{-\theta+\pi}^{\theta+\pi} \tilde{Q}_1(s)y_0(s,\rho)^{\alpha}d s= \int_{-\theta}^{\theta} \tilde{Q}_1(s+\pi)y_0(s+\pi,\rho)^{\alpha}d s=0,$$
where in the last equality we use the fact that $ \tilde{Q}_1(\theta+\pi)$ (see Proposition \ref{prop1+1}) is   odd and $y_0(s+\pi,\rho)$ is even.
Thus \begin{equation}\label{equa4-2a1}
\int_{0}^{2\pi} \bigg(\tilde{Q}_1(\theta)y_0(\theta,\rho)^{\alpha-1} \int_{0}^{\theta} \tilde{Q}_1(s)y_0(s,\rho)^{\alpha}d s\bigg)d \theta=\int_{-\pi}^{\pi} \bigg(\tilde{Q}_1(\theta+\pi)y_0(\theta+\pi,\rho)^{\alpha-1} L_1(\theta)\bigg)d \theta=0.
\end{equation}

By \eqref{equa4},
\begin{equation} \label{equa4-2}\hat{S}_1(\theta,\rho)= \int_{0}^{\theta} \big(P_1(\theta)+ \tilde{Q}_1(\theta)y_0(\theta,\rho)^{\alpha}\big)d \theta. \end{equation}

It follows from\eqref{equa7}, \eqref{equa4-2a1} and \eqref{equa4-2}   that
\begin{align} \label{equa7-2}
\begin{split}
M_{2}(\rho)&=p_{20}+\int_{0}^{2\pi} \bigg( \tilde{Q}_2(\theta)y_0(\theta,\rho)^{\alpha}+\alpha\tilde{Q}_1(\theta)y_0(\theta,\rho)^{\alpha-1} \int_{0}^{\theta}P_1(s)ds\bigg)d \theta\\
&=p_{20}+ \int_{0}^{2\pi}\tilde{Q}_2(\theta)y_0(\theta,\rho)^{\alpha}d \theta+\alpha I(\rho).
\end{split}
\end{align}
where \begin{equation}\label{equadefI} I(\rho)=\int_{0}^{2\pi}R(\theta)y_0(\theta,\rho)^{\alpha-1} d\theta,\end{equation}  with
\begin{equation}\label{equadefR}R(\theta)=\tilde{Q}_1(\theta)\int_{0}^{\theta}P_1(s)ds, \theta \in [0,2\pi].\end{equation}

In order to express  $M_{2}(\rho)$  as a linear combinations of the functions defined by \eqref{equa2-1},  we shall using the integration by part to  write $ I(\rho)$  as the form of $ \int_{0}^{2\pi} h(\theta)y_0(\theta,\rho)^{\alpha}d \theta$.

\subsubsection{Expression of $M_{2}(\rho)$ for $\theta_1\neq \pi$}

\hspace{0.3cm}

 Since $\theta_1\neq \pi$, we get from Proposition \ref{prop1+1} that,
 $$\tilde{Q}_1(\theta)=\sum_{k=1}^{m}\tilde{c}_{1k}\sin (k\theta), \ \ \mbox{with}\ \ \tilde{c}_{1k}:=\tilde{c}_{1k}^+=\tilde{c}_{1k}^-.$$

 Using  the  formula

 \begin{equation}\label{secondcheply}
\sin (k\theta)= \sin\theta U_{k-1}(\cos\theta), \end{equation}
with $U_{j}$ being  the $j$-th degree Chebyshev polynomials of the   second kind, i.e.,
\begin{equation*} \label{equaUk}
U_{j}(\cos\theta)=\sum_{i=0}^{[(j-1)/2]}(-1)^i\binom{2i+1}{j}\cos^{j-2i-1}\theta(1-\cos^{2}\theta)^i=\frac{1+(-1)^j}{2}+2\sum_{\substack{j\geq i\geq 1\\ step=-2}}\cos(i\theta),\   j=1,2,  \cdots,\end{equation*}
where``$j\geq i\geq 1,$\ step=-2" means   that $i=j, j-2,j-4, \cdots, (3+(-1)^j)/2$,
 it follows   that ${R(\theta)}/{\sin \theta}$  can be analytically extended to
$[0,2\pi]$ with
\begin{align*}
\begin{split}\frac{R(\theta)}{\sin \theta}\bigg|_{\theta = 0, 2\pi}&=\lim_{\theta \to 0, 2\pi}\frac{R(\theta)}{\sin \theta}=\lim_{\theta \to 0, 2\pi}\frac{R'(\theta)}{\cos \theta}=\lim_{\theta \to 0, 2\pi}R'(\theta)\\
&=\lim_{\theta \to 0, 2\pi}\bigg(\tilde{Q}_1(\theta)P_1(\theta)+\tilde{Q}'_1(\theta)\int_{0}^{\theta}P_1(s)ds\bigg)\\
&=\lim_{\theta \to 0, 2\pi}\tilde{Q}_1(\theta)P_1(\theta)=0,
\end{split}
\end{align*}
 here we use the fact that $p_{10}=0$.

By means of integration by parts, it follows from \eqref{equadefI} that

\begin{align} \label{equa9}
\begin{split}
\alpha I(\rho)&=\alpha\int_{0}^{2\pi}R(\theta)y_0(\theta,\rho)^{\alpha-1} d\theta=\int_{0}^{2\pi}\frac{R(\theta)}{\sin \theta}d y_0(\theta,\rho)^{\alpha}\\
&=\frac{R(\theta)}{\sin \theta} y_0(\theta,\rho)^{\alpha}\bigg|_0^{2\pi}-\int_{0}^{2\pi}\bigg(\frac{R(\theta)}{\sin \theta}\bigg)' y_0(\theta,\rho)^{\alpha}d \theta.
\end{split}
\end{align}

From  \eqref{equa7-2} and \eqref{equa9} we obtain that

\begin{equation*} \label{equa10}
M_{2}(\rho)=p_{20}+ \int_{0}^{2\pi}\bigg(\tilde{Q}_2(\theta)- \bigg(\frac{R(\theta)}{\sin \theta}\bigg)' \bigg)y_0(\theta,\rho)^{\alpha}d \theta. \end{equation*}

Let  \begin{align}\label{m2}
&m_2(\theta):=
\left\{
\begin{aligned}
m_2^+(\theta)=\tilde{Q}_2(\theta)- \bigg(\frac{R(\theta)}{\sin \theta}\bigg)'\bigg|_{\ 0 \le  \theta <  \theta_{1}},   \\
m_2^-(\theta)=\tilde{Q}_2(\theta)- \bigg(\frac{R(\theta)}{\sin \theta}\bigg)'\bigg|_{\  \theta_{1} \leq \theta \le 2\pi}.
\end{aligned}
\right.
\end{align}
Then
\begin{equation} \label{equa10-1}
M_{2}(\rho)=p_{20}+ \int_{0}^{2\pi}m_2(\theta)y_0(\theta,\rho)^{\alpha}d \theta. \end{equation}

Next we derive the expression of $\int_{0}^{\theta}P_1(s)ds$ and hence of ${R(\theta)}/{\sin \theta}$ (see \eqref{equadefR}).

 When $\theta \in (0, \theta_1)$,   we get
\begin{equation*}\label{equa13}
\int_{0}^{\theta}P_1(s)ds
={F}_1^+(\theta_1)+b_{10}^+\theta+\sum_{k=1}^{m}\frac{1}{k}\bigg(b_{1k}^+\sin (k\theta)-a_{1k}^+\cos (k\theta)\bigg).
\end{equation*}
where ${F}_1^+(\theta_1)=\sum_{k=1}^{m}\left({a_{1k}^+}/{k}\right).$

When  $\theta \in [\theta_1, 2\pi)$,   we get
\begin{align} \label{equa14}
\begin{split}
\int_{0}^{\theta}P_1(s)ds&=\int_{0}^{\theta_1}P_1(s)ds+\int_{\theta_1}^{\theta}P_1(s)ds\\
&=\int_{0}^{\theta_1}P_1(s)ds+\int_{\theta_1}^{\theta}\bigg(b_{10}^-+\sum_{k=1}^{m}a_{1k}^-\sin (ks)+b_{1k}^-\cos (ks)\bigg)ds\\
&=\int_{0}^{\theta_1}P_1(s)ds-b_{10}^-\theta_1+\sum_{k=1}^{m}\frac{1}{k}\big(a_{1k}^-\cos (k\theta_1)-b_{1k}^-\sin (k\theta_1)\big)\\
&\quad + b_{10}^-\theta+\sum_{k=1}^{m}\frac{1}{k}\big(b_{1k}^-\sin (k\theta)-a_{1k}^-\cos (k\theta)\big)  \\
&={F}_1^-(\theta_1) + b_{10}^-\theta+\sum_{k=1}^{m}\frac{1}{k}\big(b_{1k}^-\sin (k\theta)-a_{1k}^-\cos (k\theta)\big),
\end{split}\nonumber
\end{align}
where $$ {F}_1^-(\theta_1)=\int_{0}^{\theta_1}P_1(s)ds-b_{10}^-\theta_1+\sum_{k=1}^{m}\frac{1}{k}\big(a_{1k}^-\cos (k\theta_1)-b_{1k}^-\sin (k\theta_1)\big).$$

By the formula
\begin{equation*} \label{equa15}
\frac{d}{d\theta}U_{k}(\cos\theta)=-2\sum_{\substack{k\geq i\geq 1\\ step=-2}}i\sin(i\theta),\quad  k=1,2,  \cdots,\end{equation*}
where``$k\geq i\geq 1,$\ step=-2" means   that $i=k, k-2,k-4, \cdots, 1$ or $2$,
we obtain   that
\begin{equation} \label{equa17}
\begin{split}
\frac{d}{d\theta}\bigg(\frac{R(\theta)}{\sin \theta}\bigg)&=\sum_{k=2}^{m}\tilde{c}_{1k} (U_{k-1}(\cos\theta))'\int_{0}^{\theta}P_1(s)ds+\sum_{k=1}^{m}\tilde{c}_{1k} U_{k-1}(\cos\theta)P_1(\theta)\\
&=-2\sum_{k=2}^{m}\tilde{c}_{1k}\sum_{\substack{k-1\geq i\geq 1\\ step=-2}}i\sin(i\theta)\int_{0}^{\theta}P_1(s)ds+\left(\hat{c}_{11}+2\sum_{k=2}^{m}\tilde{c}_{1k}\sum_{\substack{k-1\geq i\geq 1\\ step=-2}}\cos(i\theta)\right)P_1(\theta)\\
&=-2\sum_{k=2}^{m}\tilde{c}_{1k}\sum_{\substack{k-1\geq i\geq 1\\ step=-2}}i\sin(i\theta)\left({P}_1^{\pm}(\theta_1)+b_{10}^{\pm}\theta+\sum_{\ell=1}^{m}\bigg(\frac{b_{1\ell}^{\pm}}{\ell}\sin (\ell\theta)-\frac{a_{1\ell}^{\pm}}{\ell}\cos (\ell\theta)\bigg)\right)\\
&\quad +\left(\hat{c}_{11}+2\sum_{k=2}^{m}\tilde{c}_{1k}\sum_{\substack{k-1\geq i\geq 1\\ step=-2}}\cos(i\theta)\right)\sum_{\ell=0}^{m}\big(a_{1\ell}^{{\pm}}\sin (\ell\theta) + b_{1\ell}^{{\pm}}\cos (\ell\theta)\big),
 \end{split}\end{equation}
where
\begin{equation*} \label{equa18}
\begin{split}
\hat{c}_{11}=\tilde{c}_{11}+\sum_{k=2}^{m}\frac{1-(-1)^k}{2}\tilde{c}_{1k}
 \end{split}\end{equation*}
and the superscript $^\pm$ will becomes  $^+$ (resp. $^-$)  provided $\theta\in (0,\theta_1)$  (resp. $\theta\in [\theta_1, 2\pi)$).

Furthermore, \eqref{equa17} can be written as
\begin{align} \label{equa18}
\frac{d}{d\theta}\bigg(\frac{R(\theta)}{\sin \theta}\bigg)&=-2b_{10}^{\pm}\sum_{k=2}^{m}\tilde{c}_{1k}\sum_{\substack{k-1\geq i\geq 1\\ step=-2}}i\theta\sin(i\theta)-2{P}_1^{\pm}(\theta_1)\sum_{k=2}^{m}\tilde{c}_{1k}\sum_{\substack{k-1\geq i\geq 1\\ step=-2}}i\sin(i\theta)\nonumber\\
&\quad-2\sum_{k=2}^{m}\tilde{c}_{1k}\left(\sum_{\ell=1}^{m}\bigg(\frac{b_{1\ell}^{\pm}}{\ell}\sin (\ell\theta)-\frac{a_{1\ell}^{\pm}}{\ell}\cos (\ell\theta)\bigg)\sum_{\substack{k-1\geq i\geq 1\\ step=-2}}i\sin(i\theta)\right)\nonumber\\
 &\quad +\hat{c}_{11}\sum_{\ell=0}^{m}\big(a_{1\ell}^{{\pm}}\sin (\ell\theta) + b_{1\ell}^{{\pm}}\cos (\ell\theta)\big)\nonumber\\
 &\quad +2 b_{10}^{\pm}\sum_{k=2}^{m}\tilde{c}_{1k}\sum_{\substack{k-1\geq i\geq 1\\ step=-2}}\cos(i\theta)+2\sum_{k=2}^{m}\tilde{c}_{1k}\left(\sum_{\ell=1}^{m}\big(a_{1\ell}^{{\pm}}\sin (\ell\theta) + b_{1\ell}^{{\pm}}\cos (\ell\theta)\big)\sum_{\substack{k-1\geq i\geq 1\nonumber\\ step=-2}}\cos(i\theta)\right)\nonumber\\
&=-2b_{10}^{\pm}\sum_{k=2}^{m}\tilde{c}_{1k}\sum_{\substack{k-1\geq i\geq 1\\ step=-2}}i\theta\sin(i\theta)+\hat{c}_{11}\sum_{\ell=0}^{m}\big(a_{1\ell}^{{\pm}}\sin (\ell\theta) + b_{1\ell}^{{\pm}}\cos (\ell\theta)\big)\nonumber\\
&\quad +2 b_{10}^{\pm}\sum_{k=2}^{m}\tilde{c}_{1k}\sum_{\substack{k-1\geq i\geq 1\\ step=-2}}\cos(i\theta)-2{P}_1^{\pm}(\theta_1)\sum_{k=2}^{m}\tilde{c}_{1k}\sum_{\substack{k-1\geq i\geq 1\\ step=-2}}i\sin(i\theta)\nonumber\\
&\quad+\sum_{k=2}^{m}\sum_{\ell=1}^{m}\frac{b_{1\ell}^{\pm}\tilde{c}_{1k}}{\ell}\sum_{\substack{k-1\geq i\geq 1\\ step=-2}}i\left(\cos((\ell+i)\theta)-\cos((\ell-i)\theta)\right)\\
 &\quad +\sum_{k=2}^{m}\sum_{\ell=1}^{m}\frac{a_{1\ell}^{\pm}\tilde{c}_{1k}}{\ell}\sum_{\substack{k-1\geq i\geq 1\\ step=-2}}i\left(\sin((\ell+i)\theta)-\sin((\ell-i)\theta)\right)\nonumber\\
 &\quad+\sum_{k=2}^{m}\sum_{\ell=1}^{m}{a_{1\ell}^{\pm}\tilde{c}_{1k}}\sum_{\substack{k-1\geq i\geq 1\\ step=-2}}i\left(\sin((\ell+i)\theta)+\sin((\ell-i)\theta)\right)\nonumber \\
 &\quad +\sum_{k=2}^{m}\sum_{\ell=1}^{m}{b_{1\ell}^{\pm}\tilde{c}_{1k}}\sum_{\substack{k-1\geq i\geq 1\\ step=-2}}i\left(\cos((\ell+i)\theta)+\cos((\ell-i)\theta)\right).\nonumber
 \end{align}

In what follows we will study the terms in the right hand side of \eqref{equa18}.

\begin{lem} \label{lem4} There exist a group of linear functions $ \mathcal{L}_3,\cdots, \mathcal{L}_{m-1}$ such that
\begin{equation*} \label{equa20}
\begin{split}
\sum_{k=2}^{m}\tilde{c}_{1k}\sum_{\substack{k-1\geq i\geq 1\\ step=-2}}i\theta\sin(i\theta)&=(m-1)\tilde{c}_{1m}\theta\sin((m-1)\theta) +(m-2)\tilde{c}_{1(m-1)}\theta\sin((m-2)\theta)\\
&\quad +(m-3)\left(\tilde{c}_{1(m-2)}+\mathcal{L}_3(\tilde{c}_{1m},\tilde{c}_{1(m-1)})\right)\theta\sin((m-3)\theta)
 +\cdots\\
&\quad +\left(\tilde{c}_{12}+\mathcal{L}_{m-1}(\tilde{c}_{1m},\tilde{c}_{1(m-1)},\cdots,\tilde{c}_{13})\right)\theta\sin\theta,
 \end{split}\end{equation*}
where $\mathcal{L}_j(x_1, \cdots, x_{j-1})=\lambda_{j1}x_1+\cdots +\lambda_{j(j-1)}x_{j-1}$ with $\lambda_{j1},\cdots,\lambda_{j(j-1)}\in \{0, 1\}$,    $j=3,\cdots, m-1$.
\end{lem}

\begin{proof}
By direct computations,  we obtain that when $m$ is odd,
\begin{equation*}\label{equa21}\begin{split}
&\quad \sum_{k=2}^{m}\tilde{c}_{1k}\sum_{\substack{k-1\geq i\geq 1\\ step=-2}}i\theta\sin(i\theta)\\
&=\big(\tilde{c}_{12}+\tilde{c}_{14}+\cdots+\tilde{c}_{1(m-1)}\big)\theta\sin(\theta)+2\big(\tilde{c}_{13}+\tilde{c}_{15}+\cdots+\tilde{c}_{1m}\big)\theta\sin(2\theta)\\
&\quad +3\big(\tilde{c}_{14}+\tilde{c}_{16}+\cdots+\tilde{c}_{1(m-1)}\big)\theta\sin(3\theta)+4\big(\tilde{c}_{15}+\tilde{c}_{17}+\cdots+\tilde{c}_{1m}\big)\theta\sin(4\theta)\\
&\quad +\cdots+(m-2)\tilde{c}_{1(m-1)}\theta\sin((m-2)\theta)+(m-1)\tilde{c}_{1m}\theta\sin((m-1)\theta).
\end{split}
\end{equation*}
And, when  $m$ is even,
\begin{equation*}\label{equa22}\begin{split}
&\quad \sum_{k=2}^{m}\tilde{c}_{1k}\sum_{\substack{k-1\geq i\geq 1\\ step=-2}}i\theta\sin(i\theta)\\
&=\big(\tilde{c}_{12}+\tilde{c}_{14}+\cdots+\tilde{c}_{1m}\big)\theta\sin\theta+2\big(\tilde{c}_{13}+\tilde{c}_{15}+\cdots+\tilde{c}_{1(m-1)}\big)\theta\sin(2\theta)\\
&\quad +3\big(\tilde{c}_{14}+\tilde{c}_{16}+\cdots+\tilde{c}_{1m}\big)\theta\sin(3\theta)+4\big(\tilde{c}_{15}+\tilde{c}_{17}+\cdots+\tilde{c}_{1(m-1)}\big)\theta\sin(4\theta)\\
&\quad +\cdots +(m-2)\tilde{c}_{1(m-1)}\theta\sin((m-2)\theta) +(m-1)\tilde{c}_{1m}\theta\sin((m-1)\theta).
\end{split}
\end{equation*}
Therefore, for any positive integer $m$ we have
\begin{align}\label{equa23}\begin{split}
\sum_{k=2}^{m}\tilde{c}_{1k}\sum_{\substack{k-1\geq i\geq 1\\ step=-2}}i\theta\sin(i\theta)&=(m-1)\tilde{c}_{1m}\theta\sin((m-1)\theta)+(m-2)\tilde{c}_{1(m-1)}\theta\sin((m-2)\theta)\\
&\quad +\sum_{i=1}^{m-3}i\bigg(\tilde{c}_{1(i+1)}+\tilde{c}_{1(i+3)}+\cdots+\tilde{c}_{1\alpha(m,i)}\bigg)\theta\sin(i\theta).
\end{split}
\end{align}
where $\alpha(m,i)=m-(1+(-1)^{m+i})/2$.  The conclusion follows easily from \eqref{equa23}.
\end{proof}

The next two lemmas  are aiming to study the last four terms in the right hand side of \eqref{equa18}.

\begin{lem} \label{lem5}  There exist a group of   functions $ {e}_2,\cdots, {e}_{m-1}$ and $u_0,\cdots, u_m$ such that
\begin{equation*} \label{equa25}
\begin{split}
&\quad \sum_{k=2}^{m}\sum_{\ell=1}^{m}b_{1\ell}^{\pm}\tilde{c}_{1k}\left(\frac{1}{\ell}\sum_{\substack{k-1\geq i\geq 1\\ step=-2}}i\left(\cos((\ell+i)\theta)-\cos((\ell-i)\theta)\right)+\sum_{\substack{k-1\geq i\geq 1\\ step=-2}}i\left(\cos((\ell+i)\theta)+\cos((\ell-i)\theta)\right)\right)\\
&= (m-1)\tilde{c}_{1m}\bigg[\frac{m+1}{m}b_{1m}^{\pm}\cos((2m-1)\theta)  +\left(\frac{m}{m-1}b_{1(m-1)}^{\pm}+e_2(b_{1m}^{\pm};\tilde{c}_{1m},\tilde{c}_{1(m-1)})\right)\cos((2m-2)\theta)\\
&\quad +\left(\frac{m-1}{m-2}b_{1(m-2)}^{\pm}+e_3(b_{1m}^{\pm},b_{1(m-1)}^{\pm};\tilde{c}_{1m},\tilde{c}_{1(m-1)},\tilde{c}_{1(m-2)})\right)\cos((2m-3)\theta)
 +\cdots\\
&\quad+\left(\frac{3}{2}b_{12}^{\pm}+e_{m-1}(b_{1m}^{\pm},\dots,b_{13}^{\pm};\tilde{c}_{1m},\cdots,\tilde{c}_{12})\right)\cos((m+1)\theta)\bigg]\\
&\quad+ \sum_{k=0}^{m}u_k(b_{10}^{\pm},\cdots,b_{1m}^{\pm};\tilde{c}_1,\cdots,\tilde{c}_m)\cos(k\theta),
\end{split}\end{equation*}
where $e_j(x_1,\cdots, x_{j-1};y_1,\cdots, y_j)=0$ if $y_1=\cdots=y_j=0$.
\end{lem}
\begin{proof}
Using mathematical induction, it is not hard to see by direct calculations that
\begin{align} \label{equa26}
\begin{split}
&\quad \sum_{k=2}^{m}\sum_{\ell=1}^{m}\left(\frac{b_{1\ell}^{\pm}\tilde{c}_{1k}}{\ell}\sum_{\substack{k-1\geq i\geq 1\\ step=-2}}i\cos((\ell+i)\theta)+{b_{1\ell}^{\pm}\tilde{c}_{1k}}\sum_{\substack{k-1\geq i\geq 1\\ step=-2}}i\cos((\ell+i)\theta)\right)\\
& = \sum_{k=2}^{m}\sum_{\ell=1}^{m}\frac{\ell+1}{\ell}b_{1\ell}^{\pm}\tilde{c}_{1k}\sum_{\substack{k-1\geq i\geq 1\\ step=-2}}i\cos((\ell+i)\theta)\\
&= (m-1)\tilde{c}_{1m}\bigg[\frac{m+1}{m}b_{1m}^{\pm}\cos((2m-1)\theta)+\left(\frac{m}{m-1}b_{1(m-1)}^{\pm}+e_2(b_{1m}^{\pm};\tilde{c}_{1m},\tilde{c}_{1(m-1)})\right)\cos((2m-2)\theta)\\
&\quad +\left(\frac{m-1}{m-2}b_{1(m-2)}^{\pm}+e_3(b_{1m}^{\pm},b_{1(m-1)}^{\pm};\tilde{c}_{1m},\tilde{c}_{1(m-1)},\tilde{c}_{1(m-2)})\right)\cos((2m-3)\theta) +\cdots\\
&\quad+\left(\frac{3}{2}b_{12}^{\pm}+e_{m-1}(b_{1m}^{\pm},\dots,b_{13}^{\pm};\tilde{c}_{1m},\cdots,\tilde{c}_{12})\right)\cos((m+1)\theta)\bigg]\\
&\quad+ \sum_{k=0}^{m}\tilde{u}_k(b_{10}^{\pm},\cdots,b_{1m}^{\pm};\tilde{c}_1,\cdots,\tilde{c}_m)\cos(k\theta),
\end{split}\end{align}
where $ {e}_j \ (j=2,\cdots,m-1)$ is the function of $b_{1m}^{\pm},\dots,b_{1(m+2-j)}^{\pm},\tilde{c}_{1m},\cdots,\tilde{c}_{1(m+2-j)}$ with the property that
$ {e}_j =0$ if  $\tilde{c}_{1m}=\cdots=\tilde{c}_{1(m+2-j)}=0$.   We omit the expression of these functions since it can not affect the        independence of the coefficients of
$\cos((m+1)\theta), \cdots, \cos((2m-1)\theta)$.

Obviously, the conclusion follows by observing that
\begin{equation*}
 \sum_{k=2}^{m}\sum_{\ell=1}^{m}\left(-\frac{b_{1\ell}^{\pm}\tilde{c}_{1k}}{\ell}\sum_{\substack{k-1\geq i\geq 1\\ step=-2}}i\cos((\ell-i)\theta)+{b_{1\ell}^{\pm}\tilde{c}_{1k}}\sum_{\substack{k-1\geq i\geq 1\\ step=-2}}i\cos((\ell-i)\theta)\right)
\end{equation*}
 is a linear combination of $ 1, \cos\theta, \cdots, \cos(m\theta)$  with coefficients being determined by $b_{10}^{\pm},\cdots,b_{1m}^{\pm};$ $\tilde{c}_1,$ $\cdots,$ $\tilde{c}_m$.  \end{proof}

 In a completely analogous  way  we can get the following result.
\begin{lem} \label{lem6}  There exist a group of   functions $ {f}_2,\cdots, {f}_{m-1}$ and $v_0,\cdots, v_m$ such that
\begin{equation*} \label{equa25}
\begin{split}
&\quad \sum_{k=2}^{m}\sum_{\ell=1}^{m}a_{1\ell}^{\pm}\tilde{c}_{1k}\left(\frac{1}{\ell}\sum_{\substack{k-1\geq i\geq 1\\ step=-2}}i\left(\sin((\ell+i)\theta)-\sin((\ell-i)\theta)\right)+\sum_{\substack{k-1\geq i\geq 1\\ step=-2}}i\left(\sin((\ell+i)\theta)+\sin((\ell-i)\theta)\right)\right)\\
&= (m-1)\tilde{c}_{1m}\bigg[\frac{m+1}{m}a_{1m}^{\pm}\sin((2m-1)\theta)+\left(\frac{m}{m-1}a_{1(m-1)}^{\pm}+f_2(a_{1m}^{\pm};\tilde{c}_{1m},\tilde{c}_{1(m-1)})\right)\sin((2m-2)\theta)\\
&\quad +\left(\frac{m-1}{m-2}a_{1(m-2)}^{\pm}+f_3(a_{1m}^{\pm},a_{1(m-1)}^{\pm};\tilde{c}_{1m},\tilde{c}_{1(m-1)},\tilde{c}_{1(m-2)})\right)\sin((2m-3)\theta) +\cdots\\
&\quad+\left(\frac{3}{2}a_{12}^{\pm}+f_{m-1}(a_{1m}^{\pm},\dots,a_{13}^{\pm};\tilde{c}_{1m},\cdots,\tilde{c}_{12})\right)\sin((m+1)\theta)\bigg]\\
&\quad+ \sum_{k=1}^{m}v_k(a_{10}^{\pm},\cdots,a_{1m}^{\pm};\tilde{c}_1,\cdots,\tilde{c}_m)\sin(k\theta),
\end{split}\end{equation*}
where $f_j(x_1,\cdots, x_{j-1};y_1,\cdots, y_j)=0$ if $y_1=\cdots=y_j=0$.
\end{lem}

From  equality \eqref{m2}, \eqref{equa18} and Lemma  \ref{lem4}-\ref{lem6}, we can write $m_2(\theta)$ as the linear combination of the sine and cosine functions.

\begin{prop} \label{prop3}  There exists a group of independent parameters  $$\beta_0^{\pm}; a_0^{\pm},\cdots,a_{2m-1}^{\pm}; b_0^{\pm},\cdots,b_{2m-1}^{\pm}; c_2,\cdots,c_{m}$$ such that
\begin{equation} \label{equa27}
\begin{split}
&\quad m_2^{\pm}(\theta)=\beta_0^{\pm}\sum_{k=1}^{m-1} c_{k+1} \theta\sin(k\theta)+\sum_{k=0}^{m}\left(a_{k}^{\pm}\sin(k\theta)+b_{k}^{\pm}\cos(k\theta)\right)+ c_{m} \sum_{k=m+1}^{2m-1}\left(a_{k}^{\pm}\sin(k\theta)+b_{k}^{\pm}\cos(k\theta)\right).
\end{split}\end{equation}
\end{prop}

\begin{proof}  By \eqref{m2}, $m_2(\theta)=\tilde{Q}_2(\theta)-(R(\theta)/\sin \theta)'$.
Using the  symbols in Lemma  \ref{lem4}, Lemma \ref{lem5} and  Lemma \ref{lem6}, we let
\begin{equation*} \label{equa28}
\begin{split}
&\beta_0^{\pm}=2 b_{10}^{\pm};\ c_m=(m-1)\tilde{c}_{1m},\ c_{m-1}=(m-2)\tilde{c}_{1(m-1)},\\
&  c_{k}=(k-1)\tilde{c}_{1k}+\mathcal{L}_{m+1-k}(\tilde{c}_{1m},\cdots, \tilde{c}_{1(k+1)}),\ k=2,\cdots,m-2;\\
& a_{2m-1}^{\pm}= -\frac{m+1}{m}a_{1m}^{\pm}, \
  b_{2m-1}^{\pm}=-\frac{m+1}{m}b_{1m}^{\pm}, \\
&  a_{k}^{\pm}= -\frac{k+2-m}{k+1-m}a_{1(k+1-m)}^{\pm}-f_{2m-k}\left(a_{1m}^{\pm},\cdots,a_{1(k+2-m)}^{\pm};c_{1m}^{\pm},\cdots,c_{1(k+1-m)}^{\pm}\right),         \\
&  b_{k}^{\pm}= -\frac{k+2-m}{k+1-m}b_{1(k+1-m)}^{\pm}-e_{2m-k}\left(b_{1m}^{\pm},\cdots,b_{1(k+2-m)}^{\pm};c_{1m}^{\pm},\cdots,c_{1(k+1-m)}^{\pm}\right),
\end{split}\end{equation*}
 {for}  $\ k=m+1,\cdots, 2m-2$.
Further, let

\begin{equation*} \label{equa29}
\begin{split}
&\hat{c}_{11}\sum_{\ell=0}^{m}\big(a_{1\ell}^{{\pm}}\sin \ell\theta + b_{1\ell}^{{\pm}}\cos \ell\theta\big)+2 b_{10}^{\pm}\sum_{k=2}^{m}\tilde{c}_{1k}\sum_{\substack{k-1\geq i\geq 1\\ step=-2}}\cos(i\theta)-2{P}_1^{\pm}(\theta_1)\sum_{k=2}^{m}\tilde{c}_{1k}\sum_{\substack{k-1\geq i\geq 1\\ step=-2}}i\sin(i\theta)\\
&\quad+\sum_{k=0}^{m}u_k(b_{10}^{\pm},\cdots,b_{1m}^{\pm};\tilde{c}_1,\cdots,\tilde{c}_m)\cos(k\theta)+ \sum_{k=1}^{m}v_k(a_{10}^{\pm},\cdots,a_{1m}^{\pm};\tilde{c}_1,\cdots,\tilde{c}_m)\sin(k\theta)\\
&=\sum_{k=0}^{m}\bar{u}_k(b_{10}^{\pm},\cdots,b_{1m}^{\pm};\tilde{c}_1,\cdots,\tilde{c}_m)\cos(k\theta)+ \sum_{k=1}^{m}\bar{v}_k(\theta_1;a_{10}^{\pm},\cdots,a_{1m}^{\pm};b_{10}^{\pm},\cdots,b_{1m}^{\pm};\tilde{c}_1,\cdots,\tilde{c}_m)\sin(k\theta)
\end{split}\end{equation*}
and let
\begin{equation*} \label{equa30}
\begin{split}
&a_k^{\pm}=\tilde{c}_{2k}^{\pm}-\bar{v}_k(\theta_1;a_{10}^{\pm},\cdots,a_{1m}^{\pm};b_{10}^{\pm},\cdots,b_{1m}^{\pm};\tilde{c}_1,\cdots,\tilde{c}_m),\\
&b_k^{\pm}=\tilde{d}_{2k}^{\pm}-\bar{u}_k(b_{10}^{\pm},\cdots,b_{1m}^{\pm};\tilde{c}_1,\cdots,\tilde{c}_m),
\end{split}\end{equation*}
for $ k=0,1,\cdots, m$.

The conclusion  follows from \eqref{m2},  \eqref{equa18}, Lemma  \ref{lem4}, Lemma \ref{lem5} and  Lemma\ref{lem6} directly.
\end{proof}

With the above preparation,  we are able to write the second order Melnikov functions as a linear combination of the functions defined by \eqref{equa2-1}.
First of all, from \eqref{equa10-1}   and \eqref{equa27}, in order to produce as more as possible  zeros of $M_2(\rho)$, it is natural to assume that $c_m\neq 0$.
Thus, we make the following convention.

{\bf {Convention}.} In the rest of this paper assume that $c_{1m}^+=c_{1m}^-\neq 0$.

\begin{prop} \label{prop5}  Assume  that   $M_1(\rho)\equiv 0$.  Then  there exists a group of independent parameters
\begin{equation} \label{freepar-1}\beta_1, \beta_2; \eta_1,\cdots,\eta_{m-1};\lambda_1, \cdots,\lambda_{2m-1}; \mu_1^{(1)}, \cdots,\mu_{2m-1}^{(1)};  \mu_1^{(2)}, \cdots,\mu_{2m-1}^{(2)},\end{equation}
  such that
\begin{equation} \label{equam2pro-3}
\begin{split}
M_2(\rho)&=p_{20}+\sum_{k=1}^{m-1}\eta_{k}\left(\beta_1\mathcal{D}_{k}^{E_1}(\rho,\alpha)+\beta_2\mathcal{D}_{k}^{E_2}(\rho,\alpha)\right)\\
&+\sum_{k=1}^{2m-1}\lambda_k\mathcal{S}_{k}^{E_1}(\rho,\alpha)  -\pi\beta_2\sum_{k=1}^{m-1}\eta_{k}\mathcal{S}_{k}^{E_2}(\rho,\alpha)
+\sum_{k=0}^{2m-1}\left(\mu_k^{(1)}\mathcal{C}_{k}^{E_1}(\rho,\alpha)+\mu_k^{(2)}\mathcal{C}_{k}^{E_2}(\rho,\alpha)\right),
\end{split}\end{equation}
where $E_1=[0,\theta_1]$, $E_2=[\theta_1, \pi]$ when $\theta_1\in (0,\pi)$, and  $E_1=[0,2\pi-\theta_1]$, $E_2=[2\pi-\theta_1, \pi]$ when $\theta_1\in (\pi, 2\pi)$.
\end{prop}

\begin{proof} Without loss of generality we assume that $c_m=-\alpha(m-1)\tilde{c}_{1m}=1$.     From \eqref{equa10-1} and Proposition \ref{prop3}, it follows that
  \begin{equation} \label{equam2pro-1}
\begin{split}
M_2(\rho)&=p_{20}+\sum_{k=1}^{m-1}c_{k+1}\left(\beta_0^{+}\mathcal{D}_{k}^{[0,\theta_1]}(\rho,\alpha)+\beta_0^{-}\mathcal{D}_{k}^{[\theta_1,2\pi]}(\rho,\alpha)\right)\\
&+\sum_{k=0}^{2m-1}\left(a_{k}^{+}\mathcal{S}_{k}^{[0,\theta_1]}(\rho,\alpha)+a_{k}^{-}\mathcal{S}_{k}^{[\theta_1,2\pi]}(\rho,\alpha)+b_{k}^{+}\mathcal{C}_{k}^{[0,\theta_1]}(\rho,\alpha)+b_{k}^{-}\mathcal{C}_{k}^{[\theta_1,2\pi]}(\rho,\alpha)\right),
\end{split}\end{equation}

Suppose first that  $\theta_1\in (0,\pi]$, then by letting   $E_1=[0,\theta_1]$, $E_2=[\theta_1, \pi]$, it follows from \eqref{equam2pro-1} and Lemma \ref{lemforSCD} that
 \begin{equation} \label{equam2pro-2}
\begin{split}
M_2(\rho)&=p_{20}+\sum_{k=1}^{m-1}c_{k+1}\left((\beta_0^{+}+\beta_0^{-})\mathcal{D}_{k}^{E_1}(\rho,\alpha)+2\beta_0^{-}\mathcal{D}_{k}^{E_2}(\rho,\alpha)\right)\\
&+\sum_{k=1}^{m-1}\left((a_{k}^{+}-a_{k}^{-}-2\pi\beta_0^{-}c_{k+1})\mathcal{S}_{k}^{E_1}(\rho,\alpha)-2\pi\beta_0^{-}c_{k+1}\mathcal{S}_{k}^{E_2}(\rho,\alpha)\right)\\
&+\sum_{k=m}^{2m-1}(a_{k}^{+}-a_{k}^{-})\mathcal{S}_{k}^{E_1}(\rho,\alpha)
+\sum_{k=0}^{2m-1}\left((b_{k}^{+}+b_{k}^{-})\mathcal{C}_{k}^{E_1}(\rho,\alpha)+2b_{k}^{-}\mathcal{C}_{k}^{E_2}(\rho,\alpha)\right).
\end{split}\end{equation}
In order to simplify the notations, we let $$\beta_1=\beta_0^{+}+\beta_0^{-}, \ \ \beta_2=2\beta_0^{-};$$  $$\eta_k={c_{k+1}}, \lambda_k=a_{k}^{+}-a_{k}^{-}-2\pi\beta_0^{-}c_{k+1} \ \mbox{ for}\ k=1, \cdots, m-1;$$
$$\lambda_k=a_{k}^{+}-a_{k}^{-} \ \mbox{ for}\ k=m, \cdots, 2m-1;$$
 $$\mu_k^{(1)}=b_{k}^{+}+b_{k}^{-},   \mu_k^{(2)}=2b_{k}^{-}  \ \mbox{ for}\  k=0, 1, \cdots, 2m-1. $$
 It is not hard to see that,
$$\beta_1, \beta_2; \eta_1,\cdots,\eta_{m-1};\lambda_1, \cdots,\lambda_{2m-1}; \mu_1^{(1)}, \cdots,\mu_{2m-1}^{(1)};  \mu_1^{(2)}, \cdots,\mu_{2m-1}^{(2)}$$
are independent.
Then, \eqref{equam2pro-2} turns out to be   \eqref{equam2pro-3}.

For the case that $\theta_1\in (\pi, 2\pi)$, the proof can be done exactly in the way of the case that $\theta_1\in (0,\pi]$. We omit the detail for the sake of brevity. This finish the proof.
\end{proof}

\begin{prop} \label{pro5cor}  Suppose that   $M_1(\rho)\equiv 0$ and that $\theta_1=2\pi$. Let   $E=[0,\pi]$, then
     there exists a group of   independent parameters
  $\mu_0,  \mu_1, \cdots,\mu_{2m-1}$,     such that
\begin{equation*} \label{equam2cor-2}
M_2(\rho)=2\pi b_{20}^++\sum_{k=0}^{2m-1}\mu_{k}\mathcal{C}_{k}^{E}(\rho,\alpha).
\end{equation*}
\end{prop}

\begin{proof}
 Since $M_1(\rho)\equiv 0$ and  $\theta_1=2\pi$, we have
 $2\pi b_{10}^+=p_{10}=0$, this implies that $\beta_0^+=2 b_{10}^+=0$.
Thus,  from \eqref{equam2pro-1} it yields  that
\begin{equation*} \label{equam2pro-1-2}
M_2(\rho)=p_{20}+\sum_{k=0}^{2m-1}\left(a_{k}^{+}\mathcal{S}_{k}^{[0,\theta_1]}(\rho,\alpha)+b_{k}^{+}\mathcal{C}_{k}^{[0,\theta_1]}(\rho,\alpha)\right)=2\pi b_{20}^++\sum_{k=0}^{2m-1}2b_{k}^{+}\mathcal{C}_{k}^{[0,\pi]}(\rho,\alpha).
\end{equation*}
The conclusion follows immediately  by letting $\mu_k=2b_{k}^{+}$,  $k=0,1,\cdots, 2m-1$. \end{proof}

\subsubsection{Expression of $M_{2}(\rho)$ for $\theta_1=\pi$}

\hspace{0.3cm}

In case that $\theta_1=\pi$, we can not use the integration by parts for  \eqref{equadefI} directly  since $R(\theta)/\sin \theta$ is not continuously differentiable on $[0, 2\pi]$.
Due to the  technical difficulties,  throughout  this subsection, we will impose the following condition

{(\bf{H})   }  $a_{1k}^+=a_{1k}^-$ for $k=1, 2, \cdots, m$.

\begin{prop} \label{cor1}   Suppose that   $M_1(\rho)\equiv 0$ and that $\theta_1=\pi$. Let   $E=[0,\pi]$, then, under the assumption {(\bf{H})},  there exists a group of independent parameters
\begin{equation*} \label{freepar-1}\varpi_1,  \varpi_2; \lambda_1, \cdots,\lambda_{2m-1}; \mu_0, \mu_1, \cdots,\mu_{2m-1},\end{equation*}
   such that
\begin{equation*} \label{equam2cor-1}
M_2(\rho)=p_{20}
+\varpi_1 \rho^{\alpha}+\varpi_2(\rho+2)^{\alpha}+\sum_{k=1}^{2m-1}\lambda_k\mathcal{S}_{k}^{E}(\rho,\alpha)
+\sum_{k=0}^{2m-1}\mu_k\mathcal{C}_{k}^{E}(\rho,\alpha).
\end{equation*}
\end{prop}

\begin{proof}
Denote by $P_{\mbox{\small{o}}}(\theta)$ and $P_{\mbox{\small{e}}}(\theta)$
  the odd part and the even part  of $\int_{0}^{\theta+\pi}P_1(s)ds$ respectively, then
   \begin{equation*} \int_{0}^{\theta+\pi}P_1(s)ds=P_{\mbox{\small{o}}}(\theta)+P_{\mbox{\small{e}}}(\theta)=\frac{1}{2}\int_{-\theta}^{\theta}P_1(s+\pi)ds+P_{\mbox{\small{e}}}(\theta).\end{equation*}
From \eqref{equadefI},  \eqref{equadefR}, and notice that $\tilde{Q}_1(\theta+\pi)$ is odd in $[-\pi, \pi]$,  we have that
\begin{equation} \label{cor1equa-1}
\begin{split}
I(\rho)&=\int_{0}^{2\pi}R(\theta)\left(y_0(\theta,\rho)\right)^{\alpha-1} d\theta\\
&=\int_{-\pi}^{\pi}\tilde{Q}_1(\theta+\pi)\left(y_0(\theta+\pi,\rho)\right)^{\alpha-1}\int_{0}^{\theta+\pi}P_1(s)ds d\theta\\
&=2\int_{-\pi}^{0}\tilde{Q}_1(\theta+\pi)P_{\mbox{\small{o}}}(\theta)\left(y_0(\theta+\pi,\rho)\right)^{\alpha-1} d\theta.
\end{split}\end{equation}
Using the condition {(\bf{H})}, we get by  a direct  calculation that
\begin{equation} \label{poequa-1}P_{\mbox{\small{o}}}(\theta)=\frac{1}{2}(b_{10}
^++b_{10}^-)\theta-\frac{1}{2}\sum_{k=1}^m\frac{b_{1k}
^++b_{1k}^-}{k}\sin (k\theta).\end{equation}

 Since   \begin{equation}\label{b1+--0} P_{\mbox{\small{o}}}(\pi)=\frac{1}{2}\int_{0}^{2\pi}P_1(s)ds=p_{10}=0,\end{equation}
 it follows from \eqref{poequa-1} that  $b_{10}
^++b_{10}^-=0$.
Therefore,   $P_{\mbox{\small{o}}}(\theta)/\sin \theta$, and hence $P_{\mbox{\small{o}}}(\theta)/\sin (\theta+\pi)$,
is  a trigonometric polynomial  of degree no more than $m-1$. This means that
 $\left(\tilde{Q}_1(\theta+\pi)P_{\mbox{\small{o}}}(\theta)/\sin (\theta+\pi)\right)'$ is a  trigonometric polynomial  of degree no more than $2m-1$. Therefore,
  we can get by  \eqref{cor1equa-1}   that
 \begin{equation} \label{cor1equa-2}
\begin{split}
I(\rho)&=2\int_{-\pi}^{0}\frac{\tilde{Q}_1(\theta+\pi)P_{\mbox{\small{o}}}(\theta)}{\sin (\theta+\pi)}dy_0^{\alpha}(\theta+\pi,\rho)\\
&=\frac{2\tilde{Q}_1(\theta+\pi)P_{\mbox{\small{o}}}(\theta)}{\sin (\theta+\pi)}y_0^{\alpha}(\theta+\pi,\rho)\bigg|_{-\pi}^0-2\int_{-\pi}^{0}\bigg(\frac{\tilde{Q}_1(\theta+\pi)P_{\mbox{\small{o}}}(\theta)}{\sin (\theta+\pi)}\bigg)' y_0^{\alpha}(\theta+\pi,\rho)d \theta\\
&=\frac{2\tilde{Q}_1(\theta)P_{\mbox{\small{o}}}(\theta-\pi)}{\sin \theta}y_0^{\alpha}(\theta,\rho)\bigg|_{0}^{\pi}-2\int_{0}^{\pi}\bigg(\frac{\tilde{Q}_1(\theta)P_{\mbox{\small{o}}}(\theta-\pi)}{\sin \theta}\bigg)' y_0^{\alpha}(\theta,\rho)d \theta\\
&= \frac{2\tilde{Q}_1(\theta)P_{\mbox{\small{o}}}(\theta-\pi)}{\sin \theta}y_0^{\alpha}(\theta,\rho)\bigg|_{0}^{\pi}+\sum_{k=0}^{2m-1}\omega_k\mathcal{C}_{k}^{E}(\rho,\alpha)+\sum_{k=1}^{2m-1}\sigma_k\mathcal{S}_{k}^{E}(\rho,\alpha)\\
& = \varpi_1 \rho^{\alpha}+\varpi_2(\rho+2)^{\alpha} +\sum_{k=0}^{2m-1}\omega_k\mathcal{C}_{k}^{E}(\rho,\alpha)+\sum_{k=1}^{2m-1}\sigma_k\mathcal{S}_{k}^{E}(\rho,\alpha),\end{split}\end{equation}
where $ \varpi_1,  \varpi_2; \omega_1,\cdots, \omega_{2m-1};\sigma_0,\cdots,\sigma_{2m-1}$ are the real numbers being dependent on only the coefficients of $P_1(\theta)$ and $Q_1(\theta)$. In what follows we will
show that this group of numbers are independent.

With the equalities  \eqref{poequa-1}  and  \eqref{b1+--0},   we can obtain    that
\begin{equation} \label{poequa-12}P_{\mbox{\small{o}}}(\theta-\pi)=\sum_{k=1}^ms_k\sin (k\theta),\end{equation}
where
$$s_k=\frac{1}{2}(-1)^{k+1}\frac{b_{1k}
^++b_{1k}^-}{k}, \ k=1,\cdots, m.$$

In what follows we   choose  the values of $b_{11}^{\pm},  b_{12}^{\pm}, \cdots, b_{1 (m-2)}^{\pm}$   such that   $s_1=\cdots=s_{m-2}=0$.
Using the formula \eqref{secondcheply},
 we get from \eqref{poequa-12} and Proposition  \ref{prop1+1} that for $ \theta\in [0,\pi]$,

\begin{equation*} \label{poequa-14}
\begin{split}
& \frac{\tilde{Q}_1(\theta)P_{\mbox{\small{o}}}(\theta-\pi)}{\sin (\theta)}\\
&=\left(\sum_{k=1}^{m}\tilde{c}_{1k}\sin (k\theta)+\sum_{k=0}^m\tilde{d}_{1k}\cos (k\theta)\right)\left(s_{m}U_{m-1}(\cos \theta)+s_{m-1}U_{m-2}(\cos \theta)\right)\\
&=2s_{m}\sum_{k=1}^{m}\tilde{c}_{1k}\sum_{\substack{m-1\geq i\geq 1\\ step=-2}}\sin (k\theta)\cos(i\theta)+2s_{m-1}\sum_{k=1}^{m}\tilde{c}_{1k}\sum_{\substack{m-2\geq i\geq 1\\ step=-2}}\sin (k\theta)\cos(i\theta)\\
&\quad +2s_{m}\sum_{k=0}^m\tilde{d}_{1k}\sum_{\substack{m-1\geq i\geq 1\\ step=-2}}\cos (k\theta)\cos(i\theta) +2s_{m-1}\sum_{k=0}^m\tilde{d}_{1k}\sum_{\substack{m-2\geq i\geq 1\\ step=-2}}\cos (k\theta)\cos(i\theta)\\
&\quad +H_m(\cos\theta,\sin\theta)\\
&=s_{m}\sum_{k=2}^{m}\tilde{c}_{1k}\sum_{\substack{m-1\geq i\geq m+1-k\\ step=-2}}\sin ((k+i)\theta)+s_{m-1}\sum_{k=3}^{m}\tilde{c}_{1k}\sum_{\substack{m-2\geq i\geq m+1-k\\ step=-2}}\sin ((k+i)\theta)\\
&\quad+ s_{m}\sum_{k=2}^m\tilde{d}_{1k}\sum_{\substack{m-1\geq i\geq m+1-k\\ step=-2}}\cos ((k+i)\theta)+s_{m-1}\sum_{k=3}^m\tilde{d}_{1k}\sum_{\substack{m-2\geq i\geq m+1-k\\ step=-2}}\cos ((k+i)\theta)\\
&\quad  +\bar{H}_m(\cos\theta,\sin\theta),\\
\end{split}\end{equation*}
where $H_m(\cos\theta,\sin\theta), \bar{H}_m(\cos\theta,\sin\theta)\in  \langle 1, \cos\theta, \cdots, \cos(m\theta), \sin \theta, \cdots, \sin(m\theta)\rangle$, the linear space spanned by the functions $ 1, \cos\theta, \cdots, \cos(m\theta), \sin \theta, \cdots, \sin(m\theta)$.
A further calculation leads to

\begin{equation*}\label{poequa-15}
\begin{split}&\quad s_{m}\sum_{k=2}^{m}\tilde{c}_{1k}\sum_{\substack{m-1\geq i\geq m+1-k\\ step=-2}}\sin ((k+i)\theta)+s_{m-1}\sum_{k=3}^{m}\tilde{c}_{1k}\sum_{\substack{m-2\geq i\geq m+1-k\\ step=-2}}\sin ((k+i)\theta)\\
&= s_{m}\tilde{c}_{1m}\sin((2m-1)\theta)+\left( s_{m}\tilde{c}_{1(m-1)}+\chi^{(1)}_1{(\tilde{c}_{1m}}, s_{m-1})\right)\sin((2m-2)\theta)+\cdots\\
&\quad+\left( s_{m}\tilde{c}_{12}+\chi^{(1)}_{m-2}{(\tilde{c}_{13},\cdots,\tilde{c}_{1m}}, s_{m-1})\right)\sin((m+1)\theta)), \end{split}\end{equation*}
and
\begin{equation*}\label{poequa-15}
\begin{split}\quad & s_{m}\sum_{k=2}^m\tilde{d}_{1k}\sum_{\substack{m-1\geq i\geq m+1-k\\ step=-2}}\cos ((k+i)\theta)+s_{m-1}\sum_{k=3}^m\tilde{d}_{1k}\sum_{\substack{m-2\geq i\geq m+1-k\\ step=-2}}\cos ((k+i)\theta)\\
&= s_{m}\tilde{d}_{1m}\cos((2m-1)\theta)+\left( s_{m}\tilde{d}_{1(m-1)}+\chi^{(2)}_1{(\tilde{d}_{1m}}, s_{m-1})\right)\cos((2m-2)\theta)+\cdots\\
&\quad+\left( s_{m}\tilde{d}_{12}+\chi^{(2)}_{m-2}{(\tilde{d}_{13},\cdots,\tilde{d}_{1m}}, s_{m-1})\right)\cos((m+1)\theta), \end{split}\end{equation*}
where $\chi^{(i)}_1, \cdots, \chi^{(i)}_{m-2}$ are the fixed  functions,   $i=1, 2$.  Hence,
\begin{equation} \label{poequa-16}
\left(\frac{\tilde{Q}_1(\theta)P_{\mbox{\small{o}}}(\theta-\pi)}{\sin (\theta)}\right)'=\sum_{k=m+1}^{2m-1}\left(\omega_k\cos(k\theta)+\sigma_k\sin(k\theta)\right)+\bar{H}'_m(\cos\theta,\sin\theta),  \end{equation}
where  $$\omega_k=k(s_{m}\tilde{c}_{1(k+1-m)}+\chi^{(1)}_{2m-k-1}(\tilde{c}_{1(k+2-m)},\cdots,\tilde{c}_{1m}, s_{m-1}), $$
$$ \sigma_k=-k( s_{m}\tilde{d}_{1(k+1-m)}+\chi^{(2)}_{2m-k-1}(\tilde{d}_{1(k+2-m)},\cdots,\tilde{d}_{1m}, s_{m-1}),$$
for  $ k=m+1,\cdots, 2m-2$. It is obviously that  $\omega_{m+1},\cdots, \omega_{2m-1};\sigma_{m+1},\cdots, \sigma_{2m-1}$ is a group of independent parameters whenever $s_m\neq 0$.

Using \eqref{poequa-16} and
 \eqref{cor1equa-2}, it can easily be seen that
\begin{equation*} \label{poequa-17}
\begin{split}
I(\rho)&=\tilde{Q}_1(\theta)U_{m-1}(\cos \theta)y_0^{\alpha}(\theta,\rho)\bigg|_{0}^{\pi}-2\sum_{k=0}^{2m-1}\left(\sigma_k\mathcal{S}_{k}^{E}(\rho,\alpha)+\omega_k\mathcal{C}_{k}^{E}(\rho,\alpha)\right),
\end{split}\end{equation*}
  where $\omega_0,\cdots, \omega_{m};\sigma_0,\cdots,\sigma_{m}$ are the real numbers being dependent on only the coefficients of $P_1(\theta)$ and $Q_1(\theta)$.

On the other hand, we get
from  \eqref{equa7-2} and Lemma \ref{lemforSCD} that
$$M_{2}(\rho)=p_{20}+\sum_{k=0}^{m}\left((\tilde{c}_{2k}^{+}-\tilde{c}_{2k}^{-})\mathcal{S}_{k}^{E}(\rho,\alpha)+(\tilde{d}_{2k}^{+}+\tilde{d}_{2k}^{-})\mathcal{C}_{k}^{E}(\rho,\alpha)\right)+\alpha I(\rho).$$
Consequently,
\begin{equation*} \label{poequa-17}\begin{split}
M_{2}(\rho)&=p_{20}+\alpha\tilde{Q}_1(\theta)U_{m-1}(\cos \theta)y_0^{\alpha}(\theta,\rho)\bigg|_{0}^{\pi}+\sum_{k=1}^{m}(\tilde{c}_{2k}^{+}-\tilde{c}_{2k}^{-}-2\alpha \sigma_k)\mathcal{S}_{k}^{E}(\rho,\alpha)\\
&\quad +\sum_{k=0}^{m}(\tilde{d}_{2k}^{+}+\tilde{d}_{2k}^{-}-2\alpha \omega_k)\mathcal{C}_{k}^{E}(\rho,\alpha)
-2\alpha\sum_{k=m+1}^{2m-1}\left(\sigma_k\mathcal{S}_{k}^{E}(\rho,\alpha)+\omega_k\mathcal{C}_{k}^{E}(\rho,\alpha)\right)\\
&=p_{20}+\varpi_1y_0^{\alpha}(0,\rho)+\varpi_2y_0^{\alpha}(\pi,\rho)
+\sum_{k=1}^{2m-1}\lambda_k\mathcal{S}_{k}^{E}(\rho,\alpha)
+\sum_{k=0}^{2m-1}\mu_k\mathcal{C}_{k}^{E}(\rho,\alpha)\\
&=p_{20}+\varpi_1 \rho^{\alpha}+\varpi_2(\rho+2)^{\alpha}
+\sum_{k=1}^{2m-1}\lambda_k\mathcal{S}_{k}^{E}(\rho,\alpha)
+\sum_{k=0}^{2m-1}\mu_k\mathcal{C}_{k}^{E}(\rho,\alpha),\end{split}
\end{equation*}
where $$\varpi_1=-\alpha\left(ms_m+(m-1)s_{m-1}\right)\sum_{k=0}^{m}\tilde{d}_{1k}, \quad  \varpi_2=\alpha(ms_m-(m-1)s_{m-1})\sum_{k=0}^{m}(-1)^{m+k}\tilde{d}_{1k},$$
$$\lambda_k=\tilde{c}_{2k}^{+}-\tilde{c}_{2k}^{-}-2\alpha \sigma_k\ \mbox{ for} \ k=1,\cdots,m, \quad \mu_k=\tilde{d}_{2k}^{+}+\tilde{d}_{2k}^{-}-2\alpha \omega_k \ \mbox{ for} \ k=0,1,\cdots,m,$$   and  $\lambda_k=-2\alpha \sigma_k$, $\mu_k=-2\alpha \omega_k$ for $k=m+1,\cdots, 2m-1$.
It is clear that $ \lambda_1, \cdots,\lambda_{2m-1}; \mu_0, \mu_1, \cdots,\mu_{2m-1}$ is a group of independent parameters.

 The proof finishes observing that the real numbers $\varpi_1, \varpi_2,; \lambda_1, \cdots,\lambda_{2m-1}; \mu_0, \mu_1, \cdots,\mu_{2m-1}$
are independent.

\end{proof}

Now we are in position  to explain   that why we assume throughout this paper that $ \frac{q-1}{p-1}\notin \mathbf{Z}_{\leq 1}$.
  \begin{rem} The assumption that  $ \frac{q-1}{p-1}\notin \mathbf{Z}_{\leq1}$  is  natural. In fact, if $ \frac{q-1}{p-1}\in \mathbf{Z}_{\leq 1}$, then  $\alpha=\frac{q-p}{1-p}=-\frac{q-1}{p-1}+1$   becomes  a non-negative integer. From the conclusions of this section, it follows   that the first and  second order Melnikov functions $M_i(\rho) \ (i=1,2)$
are polynomials of $\rho$. Thus the study of the  maximum number of zeros of
$M_i(\rho) (i=1,2)$  is   trivial.
\end{rem}

\section{  Chebyshev family for the first non-vanishing Melnikov function}
The proof of the main result of this paper rely on the theory of  Chebyshev system. We first recall the definitions of T , CT and ECT system, the notions of continuous,
and discrete Wronskian and a useful characterization of CT and ECT -system. For more information on this subject, the readers are referred to the monographs \cite{KS1966}.

\begin{defn} Let $ f_0,  f_1,  \cdots,  f_{n-1}$  be analytic functions on an open interval $I$.

\begin{itemize}

\item[\textsc{(}a\textsc{)}] $ (f_0,  f_1,  \cdots,  f_{n-1})$  is a Chebyshev system (T-system) on $I$  if any nontrivial linear combination
$$\alpha_0f_0(x)+\alpha_1f_1(x)+\cdots+\alpha_{n-1}f_{n-1}(x)$$
has at most $n-1$ isolated zeros on $I$ .

\item[\textsc{(}b\textsc{)}] $ (f_0,  f_1,  \cdots,  f_{n-1})$  is a complete Chebyshev system (CT-system) on   $I$  if $ (f_0,  f_1,  \cdots,  f_{k-1})$
is a T-system on $I$ for all $k=1, 2, \cdots, n$.

\item[\textsc{(}c\textsc{)}] $ (f_0,  f_1,  \cdots,  f_{n-1})$  is an extended complete Chebyshev system (ECT-system) on $I$ if, for all $k=1, 2, \cdots, n$, any nontrivial linear combination
$$\alpha_0f_0(x)+\alpha_1f_1(x)+\cdots+\alpha_{k-1}f_{k-1}(x)$$
has at most $k-1$ isolated zeros on $I$ counting multiplicity.

\end{itemize}

Generally, it is hard to judge  whether   an   ordered set of functions is
a T-system by use the definition.  The notion of the Wronskian proves to be an extremely useful tool in this problem.

\begin{defn}
  Let $f_0,\ldots,f_{k}$ be analytic functions on an open interval $I$.
  The continuous Wronskian of $(f_0,\ldots,f_{k})$ at $t\in I$ is
 \begin{align*}
   W[\bm f_k](t)=\det (f^{(i)}_j(t);\ 0\leq i,j\leq k)
   =\left|
  \begin{array}{ccccccc}
  f_0(t)             &\ldots   &f_{k}(t)\\
  f'_0(t)            &\ldots   &f'_{k}(t)   \\
  \vdots            &\ddots   &\vdots                 \\
 f^{(k)}_0(t)       &\ldots   & f^{(k)}_{k}(t)    \\
  \end{array}
  \right|.
 \end{align*}
  The discrete Wronskian of $(f_0,\ldots,f_{k})$ at $(t_0,\ldots,t_{k})\in I^{k+1}$ is
 \begin{align*}
  D[\bm f_k;\bm t_k]=\det (f_j(t_i);\ 0\leq i,j\leq k)
  =\left|
  \begin{array}{ccccccc}
  f_0(t_0)             &\ldots   &f_{k}(t_0)\\
  f_0(t_1)            &\ldots   &f_{k}(t_1)   \\
  \vdots            &\ddots   &\vdots                 \\
 f_0(t_{k})       &\ldots   & f_{k}(t_{k})    \\
  \end{array}
  \right|.
\end{align*}
Here   we use the shorthand $t_0,t_1,\cdots, t_{k-1}:=\bm t_k$  for the sake of brevity, just like  \cite{GGM2016} does.
\end{defn}

The next result is well-known  \cite{KS1966}.
\begin{lem} \label{lem0-1}
Let $ f_0,f_1,\ldots,f_{n-1} $ be analytic functions on an interval $I$ of $\mathbf{R}$. The following statements hold.
\begin{itemize}
  \item [(i)] The ordered set $(f_0,f_1,\ldots,f_{n-1})$ is a CT-system on $I$ if and only if, for each $ k=1,\ldots, n $,
  \begin{align*}
    D[\bm f_{k};\bm t_{k}]\neq0 \text{ for all } \bm t_{k}\in E^{k+1} \text{ such that } t_i\neq t_j \text{ for } i\neq j.
  \end{align*}
  \item [(ii)] The ordered set $(f_0,f_1,\ldots,f_{n-1})$ is an ECT-system on $I$ if and only if, for each $ k=1,\ldots,n $,
  \begin{align*}
    W[\bm f_{k}](t)\neq 0 \ \text{ for all }\  t\in I.
  \end{align*}
\end{itemize}
\end{lem}

\end{defn}

Let's use the above result and some other known results  to  study the Chebyshev property   of some set of functions appear in the Melnikov functions of equation \eqref{equa3} .
In order to  simplify the symbols, in the rest of this section we set
\begin{align*}
 s_k(\theta) = \sin(k\theta),  \quad  c_k(\theta) = \cos(k\theta).
\end{align*}
Furthermore, we  omit $``(\theta)"$,  i.e.,   $s_k(\theta)$ will be replaced with $s_k$.

\begin{lem} \label{lem3-1} \textsc{(}\cite{HL2022}\textsc{)} For any positive integer $n$, the  ordered set of functions
$$(1, c_1, c_2, \cdots, c_n, s_n,\cdots, s_2, s_1)$$
is an ECT-system on the open interval  $(0, \pi)$.
\end{lem}

\begin{lem}(\cite{HP2020}) \label{lem3-2} For any positive integer  $n$, the  ordered set of functions
$$(1, \theta,  s_1, c_1, \theta c_1,  s_2, c_2, \theta c_2,\cdots,  s_n,  c_n, \theta c_n),$$
is an ECT-system on the open interval  $(0, \pi)$.
\end{lem}

\begin{lem} (\cite{Kau1988}) \label{lem3-3} For any positive integer $n$, the set of ordered functions
$$(1, f_0,  f_1,  \cdots,  f_n)$$
is an ECT-system on an interval $E$ if and only if   $( f'_0,  f'_1,  \cdots,  f'_n )$ is an ECT-system on   $E$.
\end{lem}

We shall use the above lemmas to prove the following results.
\begin{prop} \label{prop3-2} For any positive integer $n$, the  ordered set of functions
\begin{equation} \label{equa40}(1, c_1, s_1,\theta s_1,  c_2, s_2, \theta s_2,\cdots,   c_n,  s_n, \theta s_n),\end{equation}
is an ECT-system on the open interval  $(0, \pi)$.
\end{prop}

\begin{proof}
If follows from Lemma \ref{lem3-2} and Lemma \ref{lem3-3} immediately
 that
  the  ordered set of functions
 \begin{equation} \label{equa41} (1, c_1, s_1, c_1-\theta s_1, c_2, s_2, c_2-2\theta s_2,
   \cdots, c_n, s_n,  c_n-n\theta s_n)\end{equation}
is an ECT-system on the open interval   $(0, \pi)$.
One can easy to   check that for each $k=1,2,\cdots, n$, the Wronskian of the first $k$ elements of the ordered set \eqref{equa41}, is equal to the one of the order set in \eqref{equa40}. Thus, by Lemma \ref{lem0-1}  the conclusion follows directly.

\end{proof}

In order to study the number of zeros of  the Melnikov function defined by \eqref{equa10-1}, we need the following Proposition.
\begin{prop} \label{prop3-3} For any   integers   $n_0, \ell_0$ and $k_0$ with $k_0>n_0\geq 1$ and $k_0>\ell_0\geq 0$, the  ordered set of functions
\begin{equation} \label{equa41} (1, c_1, s_1,\theta s_1,  c_2, s_2, \theta s_2,\cdots, c_{n_0}, s_{n_0}, \theta s_{n_0},
  c_{n_0+1},  c_{n_0+2}, \cdots, c_{k_0}, s_{k_0},\cdots, s_{\ell_0+2}, s_{\ell_0+1})
\end{equation}
is an ECT-system on the open interval  $(0, \pi)$.
\end{prop}

\begin{proof}

 For any integers  $n\geq 1, \ell\geq 0$ and $k>n, \ell$, denote by
\begin{equation} \label{equa44} F_{n, k,\ell}:=(s_1, c_1,  \theta c_1, s_2, c_2, \theta c_2,\cdots, s_{n}, c_{n}, \theta c_{n},
s_{n+1},  s_{n+2},\cdots, s_{k}, c_{k},   \cdots, c_{\ell+2}, c_{\ell+1} ).
 \end{equation}
From Lemma \ref{lem3-3}, it suffices to prove that $F_{n, k,\ell}$
is an ECT-system on the open interval  $(0, \pi)$.

By Proposition  \ref{prop3-2}, the ordered set of the first $3n+1$ elements of  \eqref{equa41}
\begin{equation*} \label{equa42} (1, c_1, s_1,\theta s_1,  c_2, s_2, \theta s_2,\cdots, c_{n}, s_{n}, \theta s_{n})
\end{equation*}
is an ECT-system on the open interval  $(0, \pi)$. Together with Lemma \ref{lem0-1}, the conclusion is true once we show that
the  Wroskian of the first $3n+1+\kappa$ elements of \eqref{equa44} is non-zero on $(0, \pi)$, where $\kappa=1,\cdots,2k-n-\ell$.
In the following we only prove for $\kappa=k-n+1, \cdots, 2k-n-\ell$, and the rest can be verified in a similar way.

We denote by $W(F_{n,k,\ell})$ the Wroskian of $F_{n,k,\ell}$. By using the  properties of determinants, we obtain by computations that

\begin{align*}
W&(F_{n,k,\ell}) \\
&=
  \left|
  \begin{array}{ccccccccccc}
   s_1         &c_1            &\theta c_1            &\cdots     s_i         &c_i            &\theta c_i            &\cdots       & s_n         &c_n               &\theta c_n   \\
  c_1         &-s_1           &(\theta c_1)'         &\cdots      ic_i         &-is_i           &(\theta c_i)'         &\cdots   & nc_n         &-ns_n             &(\theta c_n)' \\
  -s_1        &-c_1           &(\theta c_1)''        &\cdots      -i^2s_i      &-i^2c_i           &(\theta c_i)''        &\cdots  & -n^2s_n        &-n^2c_n           &(\theta c_n)''  \\
  -s_1'       &-c_1'          &(\theta c_1)'''       &\cdots      -i^2s_i'     &-i^2c_i'          &(\theta c_i)'''       &\cdots  & -n^2s_n'       &-n^2c_n'          &(\theta c_n)'''  \\
  \vdots    &\vdots           &\vdots                &\ddots      \vdots    &\vdots           &\vdots                &\ddots  &\vdots           &\vdots            &\vdots         \\
 -s_1^{(r)}   & -c_1^{(r)}    &(\theta c_1)^{(r+2)}  &\cdots      -s_i^{(r)}   & -c_i^{(r)}    &(\theta c_i)^{(r+2)}  &\cdots   &-n^2s_n^{(r)}    & -n^2c_n^{(r)}    &(\theta c_n)^{(r+2)}
  \end{array}
  \right. \\
 & \quad
  \left.
  \begin{array}{ccccccc}
\quad \quad  & s_{n+1}             &\cdots               & s_{k}              & c_{k}                &\cdots               & c_{\ell+1}   \\
\quad \quad  &  (n+1)c_{n+1}         &\cdots              & kc_{k}         & -ks_{k}          &\cdots               & -(\ell+1)s_{\ell+1}  \\
\quad \quad  &  -(n+1)^2 s_{n+1}       &\cdots            &-k^2 s_{k}      & -k^2c_{k}         &\cdots              & -(\ell+1)^2c_{\ell+1} \\
 \quad \quad  &  -(n+1)^2 s_{n+1}'      &\cdots          & -k^2 s_{k}'    & -k^2c_{k}'        &\cdots              & -(\ell+1)^2c_{\ell+1}'   \\
 \quad \quad  &   \ddots                 &\ddots           &\vdots                  &\vdots               &\ddots                &\vdots  \\
 \quad \quad  &    -(n+1)^2 s_{n+1}^{(r)}  &\cdots        &-k^2 s_{k}^{(r)}  &-k^2 c_{k}^{(r)}    &\cdots              &-(\ell+1)^2 c_{\ell+1}^{(r)}
  \end{array}
  \right|\\
  &=
  \left|
  \begin{array}{cccccc}
   s_1                &c_1             &\theta c_1                        &\cdots                     & s_i                &c_i            \\
  c_1                &-s_1             &(\theta c_1)'                     &\cdots                    & ic_i                &-is_i            \\
 (n^2-1)s_1        & (n^2-1)c_1        &-2s_1+(n^2-1)(\theta c_1)         &\cdots                    & (n^2-i^2)s_i        & (n^2-i^2)c_i       \\
  (n^2-1)s_1'       & (n^2-1)c_1'      &-2s_1'+(n^2-1)(\theta c_1)'      &\cdots                    &  (n^2-i^2)s_i'       & (n^2-i^2)c_i'       \\
  \vdots             &\vdots           &\vdots                         &\ddots                           &\vdots           &\vdots                                \\
  (n^2-1)s_1^{(r)}   &  (n^2-1)c_1^{(r)}    &-2s_1^{(r)}+(n^2-1)(\theta c_1)^{(r)}  &\cdots         & (n^2-i^2)s_i^{(r)}   &  (n^2-i^2)c_i^{(r)}
  \end{array}
  \right. \\
 & \quad
  \left.
  \begin{array}{ccccccc}
\theta c_i                                      &\cdots    & s_{n}      &c_n     &\theta c_n            & s_{n+1}             &\cdots                          \\
(\theta c_i)'                                    &\cdots    & nc_n       &-ns_n  &(\theta c_n)'         & (n+1)c_{n+1}         &\cdots                   \\
-2is_1+(n^2-i^2)(\theta c_i)                    &\cdots     & 0         & 0      &-2ns_n                &  -(2n+1) s_{n+1}       &\cdots               \\
 -2is_1'+(n^2-i^2)(\theta c_i)'                &\cdots      & 0         & 0      &-2n(s_n)'             & -(2n+1) s_{n+1}'      &\cdots                \\
\vdots                                          &\ddots      & 0        & 0      & \ddots              & \ddots                 &\ddots                          \\
 -2is_i^{(r)}+(n^2-i^2)(\theta c_i)^{(r)}        &\cdots    & 0         & 0      &-2n(s_n)^{(r)}             & -(2n+1)s_{n+1}^{(r)}  &\cdots
  \end{array}
  \right.\\
 & \quad
  \left.
  \begin{array}{cccc}
\quad \quad s_{k}              & c_{k}                &\cdots               & c_{\ell+1}   \\
\quad \quad kc_{k}            &-ks_{k}            &\cdots              & -(\ell+1)s_{\ell+1}  \\
\quad \quad (n^2-k^2)s_{k}      &(n^2-k^2)c_{k}          &\cdots               &(n^2-(\ell+1)^2)c_{\ell+1} \\
\quad \quad  (n^2-k^2) s_{k}'     & (n^2-k^2)c_{k}'           &\cdots               &  (n^2-(\ell+1)^2)c_{\ell+1}'   \\
\quad \quad \vdots                 &\vdots                 &\ddots                     &\vdots  \\
\quad \quad  (n^2-k^2)s_{k}^{(r)}  & (n^2-k^2)c_{k}^{(r)}    &\cdots              &(n^2-(\ell+1)^2) c_{\ell+1}^{(r)}
  \end{array}
  \right|\\
 & =(-1)^{[n-1]/2}\cdot \prod_{i=1}^{n-1}(n^2-i^2)^3 \prod_{i=n+1}^{k}(n^2-i^2)\cdot  \prod_{i=\ell+1}^{k}(n^2-i^2)\cdot  W(F_{n-1,k,\ell}),
\end{align*}
where $r=2n+2k-\ell-3$.  Here we use the relation $(\theta c_i)^{(j+2)}+n^2(\theta c_i)^{(j)}=(n^2-i^2)(\theta c_i)^{(j)}-2is_i^{(j)}$.

Denote by
\begin{equation*} \label{equa47} J_{n,k}=(-1)^{[n-1]/2}\cdot \prod_{i=1}^{n-1}(n^2-i^2)^3 \prod_{i=n+1}^{k}(n^2-i^2),\quad   I_{\ell,k}= \prod_{i=\ell+1}^{k}(n^2-i^2).
 \end{equation*}
From the above equality we find that
\begin{equation*}  W(F_{n,k,\ell})=J_{n,k} I_{\ell,k} W(F_{n-1,k,\ell}),
 \end{equation*}
which implies that
\begin{equation*} \label{equa48} W(F_{n,k,\ell})=\prod_{i=1}^n\left(J_{i,k} I_{\ell,k}\right) W(F_{i-1,k,\ell})=\prod_{i=1}^n\left(J_{i,k} I_{\ell,k}\right)W(s_1,  s_2, \cdots,  s_{k}, c_{k}, \cdots, c_{\ell+2}, c_{\ell+1}).
 \end{equation*}
By Lemma \ref{lem3-1} and Lemma   \ref{lem3-3}, it is easy to see that $(s_1,  s_2, \cdots,  s_{k}, c_{k}, \cdots, c_{\ell+2}, c_{\ell+1})$ is an ECT-system  on $(0,\pi)$, which implies that $W(s_1,  s_2, \cdots,  s_{k}, c_{k}, \cdots, c_{\ell+2}, c_{\ell+1})\neq 0$.
Thus, from Lemma \ref{lem0-1} we know that \eqref{equa44} is an ECT-system  on $(0,\pi)$. This complete the proof.

\end{proof}

\begin{prop} \label{prop3-5} Let $\mathcal{C}_{k}^{E}$, $\mathcal{S}_{k}^{E}$ and $\mathcal{D}_{k}^{E}$   be defined in \eqref{equa2-1} and $\hat{I}$ be the   interval   defined by  \eqref{equa1-4}. For any fixed   $\vartheta\in [0, \pi]$, set $J_1=[0, \vartheta]$, $J_2=[\vartheta, \pi]$,
the following statements are true for any fixed $\beta \notin \mathbf{Z}_{\leq 2}$.
\begin{itemize}
\item [(i)] For any integer $n>1$, the ordered set of functions
\begin{equation} \label{equa49-0}\left(\mathcal{C}_{0}^{J_1}, \cdots, \mathcal{C}_{n}^{J_1}, \mathcal{S}_{n}^{J_1}, \cdots, \mathcal{S}_{1}^{J_1},\mathcal{C}_{0}^{J_2}, \cdots, \mathcal{C}_{n}^{J_2}\right) \end{equation}
is an ECT-system on   $\hat{I}$.
\item [(ii)] For any integer $n>m$, the ordered set of functions\begin{equation} \label{equa49-1}\bigg(\mathcal{C}_{0}^{J_1}, \mathcal{C}_{1}^{J_1}, \mathcal{S}_{1}^{J_1}, \mathcal{D}_{1}^{J_1}, \cdots, \mathcal{C}_{m-1}^{J_1}, \mathcal{S}_{m-1}^{J_1}, \mathcal{D}_{m-1}^{J_1},
  \mathcal{C}_{m}^{J_1}, \cdots, \mathcal{C}_{n}^{J_1}, \mathcal{S}_{n}^{J_1}, \cdots, \mathcal{S}_{m}^{J_1},\mathcal{C}_{0}^{J_2}, \mathcal{C}_{1}^{J_2},  \cdots, \mathcal{C}_{n}^{J_2}\bigg)
  \end{equation}
is an ECT-system on  $\hat{I}$.
\item [(iii)]  For any integer $n>m$, the ordered set of functions
\begin{equation} \label{equa49-2}\begin{split}&\bigg(\mathcal{C}_{0}^{J_1}, \mathcal{C}_{1}^{J_1}, \mathcal{S}_{1}^{J_1}, \mathcal{D}_{1}^{J_1},  \cdots, \mathcal{C}_{m-1}^{J_1}, \mathcal{S}_{m-1}^{J_1}, \mathcal{D}_{m-1}^{J_1}, \mathcal{C}_{m}^{J_1}, \cdots, \mathcal{C}_{n}^{J_1}, \mathcal{S}_{n}^{J_1}, \cdots, \mathcal{S}_{m}^{J_1},\\
&\mathcal{C}_{0}^{J_2}, \mathcal{C}_{1}^{J_2}, \mathcal{S}_{1}^{J_2}, \mathcal{D}_{1}^{J_2},\cdots, \mathcal{C}_{m-1}^{J_2}, \mathcal{S}_{m-1}^{J_2}, \mathcal{D}_{m-1}^{J_2}, \mathcal{C}_{m}^{J_2}, \cdots, \mathcal{C}_{n}^{J_2}\bigg)
 \end{split}  \end{equation}
is an ECT-system on  $\hat{I}$.
\end{itemize}
Here,  we stipulate that, when  $\vartheta=0$ (resp. $\vartheta=\pi$), the functions $\mathcal{C}_{j}^{J_1}, \mathcal{S}_{j}^{J_1}, \mathcal{D}_{j}^{J_1}$ (resp. $\mathcal{C}_{j}^{J_2}, \mathcal{S}_{j}^{J_2}, \mathcal{D}_{j}^{J_2}$ ) are removed from  the ordered set \eqref{equa49-0}-\eqref{equa49-2}.
\end{prop}

\begin{proof}  For the sake of brevity we only provide the proof of statement (iii) since the proofs  of (i) and (ii) are similar.

   We first recall the  Theorem 1 of \cite{HLZ2023}:   Suppose that $E_0,\ldots, E_n$ are non-intersecting intervals, and $E$ is an open interval that contains $\bigcup^n_{i=0} E_i$ in its interior.
Suppose that $U$ is also an open interval, and there exists a $C^{\omega}$ function $G=G(t,y)$ defined on $E\times U$. Let $G_i=G|_{E_i\times U}$ for $i=0,\ldots,n$.
Then the ordered set of functions
\begin{equation}\label{eqhlz2}
\mathcal{F}=\bigcup^{n}_{i=0}\{I_{i,0},I_{i,1},\ldots,I_{i,m_i}\},
\end{equation}
where
\begin{equation}\label{eqhlz3}
I_{i,j}(y):=\int_{E_i}f_{i,j}(t)G_{i}(t,y)dt,
\end{equation}   form an ECT-system on $U$ if the following hypotheses are satisfied:
\begin{itemize}
  \item [(H.1)] For each $i\in\{0,\ldots,n\}$, the  ordered functions $(f_{i,0},\cdots,f_{i,m_i})$ form  a CT-system on $E_i$.
  \item [(H.2)] For each $y\in U$, the  ordered functions $(G,\partial_y G, \cdots,\partial^{M}_y G)$ form  a CT-system in the variable $t$ on $E$, where $M={\rm Card}(\mathcal F)-1=m_0+m_1+\cdots+m_n+n-1$.
\end{itemize}

We will apply this result to prove the conclusion (iii). First,  let $G(t,y)=(1+y-\cos t)^{\beta}:=g(t,y)^{\beta}$, $t\in (0, \pi), y\in \hat{I}$. It is clear that  $G$  is monotonic in $t$ for each fixed $y$.
For any fixed positive integer $k$,  consider the discrete Wronskian   of $\{G,\partial_y G,\cdots,\partial^{k}_y G\}$ at $(t_0,\ldots,t_{k})$
with $0<t_0<t_1<\cdots<t_{k}<\pi$.
 By direct computation we obtain that for $\beta \notin \mathbf{Z}^+$,
  \begin{equation*}
\begin{split}
  D_k(y):&=D[G,\partial_y G,\cdots,\partial^{k}_y G;t_0,\ldots,t_{k}](y) \\
 & =\beta^k(\alpha-1)^{k-1}\cdots(\beta-k+1)\left|
  \begin{array}{ccccccc}
  g^{\beta}(t_0,y)    &g^{\beta-1}(t_0,y)          &\ldots   &g^{\beta-k}(t_0,y)\\
  g^{\beta}(t_1,y)      &g^{\beta-1}(t_1,y)       &\ldots   &g^{\beta-k}(t_1,y)   \\
  \vdots            &\vdots      &\ddots   &\vdots                 \\
 g^{\beta}(t_{k},y)      &g^{\beta-1}(t_k,y)   &\ldots   & g^{\beta-k}(t_{k},y)    \\
\end{array}
  \right|\\
&= (-1)^{(k+1)k/2}\prod_{j=0}^{k-1}(\beta-j)^{k-j}\prod_{i=0}^{k}g^{\beta-k}(t_i,y)\prod_{0\leq i<j\leq k}\left(g(t_j,y)-g(t_i,y)\right)\neq 0.
\end{split}\end{equation*}

Thus, by Lemma  \ref{lem0-1}, for each $y\in \hat{I}$, the  ordered functions $\{G,\partial_y G, \cdots,\partial^{k}_y G\}$ form  a CT-system in the variable $t$ on $(0, \pi)$.

Let $E_0=J_1, E_1=J_2$, $I_{0,j}(y):=\int_{E_0}f_{0,j}(t)G(t,y)dt$  and $I_{1,j}(y):=\int_{E_1}f_{1,j}(t)G(t,y)dt$ with
\begin{equation*}\left(f_{0,0}(t),\cdots,f_{0,m+2n-1}(t)\right)=\left(1, c_1, s_1,t s_1, \cdots,   c_{m-1}, s_{m-1}, t s_{m-1},
  c_{m},   \cdots, c_{n}, s_{n},\cdots, s_{m}\right)\end{equation*}
and
\begin{equation*}\left(f_{1,0}(t),\cdots,f_{1,2m+n-2}(t)\right)=\left(1, c_1, s_1, t s_1, \cdots,   c_{m-1}, s_{m-1}, t s_{m-1}, c_m, \cdots, c_{n}\right)\end{equation*}

From Proposition  \ref{prop3-3}, $(f_{0,0}, f_{0,1}, \cdots, f_{0,m+2n-1})$ is a CT-system on $E_0$,  $(f_{1,0}, f_{0,1}, \cdots, f_{0,2m+n-2})$ is a CT-system on $E_1$.
Thus,   it follows from  Theorem 1 of \cite{HLZ2023} that the ordered set of functions
  \eqref{equa49-2} is an ECT-system on the interval  $\hat{I}$. This verify the conclusion (iii).\end{proof}

Next, in order to study the number of zeros of $M_2$ for the case that $\theta_1=\pi$, we need the following result.

\begin{prop} \label{prop3-7} Let $E=[0,\pi]$,  $\mathcal{C}_{k}^{E}$, $\mathcal{S}_{k}^{E}$    be defined in \eqref{equa2-1} and $\hat{I}$ be the   interval   defined by  \eqref{equa1-4}.  Then for any   fixed $\beta \notin \mathbf{Z}_{\leq 2}$,
the ordered set of functions
 \begin{equation} \label{equa60-1}\left(1,\rho^{\beta}, (\rho+2)^{\beta}, \mathcal{C}_{0}^{E}(\rho,\beta), \cdots, \mathcal{C}_{2m-1}^{E}(\rho,\beta), \mathcal{S}_{2m-1}^{E}(\rho,\beta), \cdots, \mathcal{S}_{1}^{E}(\rho,\beta)\right) \end{equation}
is an ECT-system on   $\hat{I}$.
\end{prop}

\begin{proof}
Denote by $\beta_i=\beta-i$ for $i=1,2,3$ and $\rho_1=\rho^{-1}$, $\rho_2=2\rho_1+1$, $\rho_3=\rho_2^{-1}$.

From Lemma \ref{lem3-3}, the ordered set of functions  \eqref{equa60-1} is  an ECT-system on   $\hat{I}$ if and only if

\begin{equation} \label{equa60-2}\left(\rho^{\beta_1}, (\rho+2)^{\beta_1}, \mathcal{C}_{0}^{E}(\rho,\beta_1), \cdots, \mathcal{C}_{2m-1}^{E}(\rho,\beta_1), \mathcal{S}_{2m-1}^{E}(\rho,\beta_1), \cdots, \mathcal{S}_{1}^{E}(\rho,\beta_1)\right) \end{equation}
is an ECT-system on   $\hat{I}$. Multiplying   by $\rho^{-\beta_1}$ each function in \eqref{equa60-2}, it turns out  that \eqref{equa60-1} is  an ECT-system on   $\hat{I}$ if and only if
\begin{equation} \label{equa60-3}\left(1, (1+2\rho_1)^{\beta_1}, \mathcal{C}_{1,0}^{E}(\rho_1,\beta_1), \cdots, \mathcal{C}_{1,2m-1}^{E}(\rho_1,\beta_1), \mathcal{S}_{1,2m-1}^{E}(\rho_1,\beta_1), \cdots, \mathcal{S}_{1,1}^{E}(\rho_1,\beta_1)\right) \end{equation}
is  an ECT-system on   $\hat{I}_1$,  where $\hat{I}_1=\{ \rho:  \rho^{-1}\in \hat{I} \}$,  and
 \begin{equation*}
\begin{split}
&\mathcal{C}_{1,k}^{E}(\rho,\beta)=\int_{E}\cos (k\theta)\left(1+\rho(1-\cos\theta)\right)^{\beta}d\theta,\ \ k=0,1, \cdots, 2m-1, \\
&\mathcal{S}_{1,k}^{E}(\rho,\beta)=\int_{E}\sin (k\theta)\left(1+\rho(1-\cos\theta)\right)^{\beta}d\theta, \ \ k=1, 2,  \cdots, 2m-1.
\end{split}
\end{equation*}

Again, by Lemma \ref{lem3-3},  \eqref{equa60-3} is  an ECT-system on   $\hat{I}_1$ if and only if
\begin{equation} \label{equa60-4}\left((1+2\rho_1)^{\beta_2}, \mathcal{C}_{2,0}^{E}(\rho_1,\beta_2), \cdots, \mathcal{C}_{2,2m-1}^{E}(\rho_1,\beta_2), \mathcal{S}_{2,2m-1}^{E}(\rho_1,\beta_2), \cdots, \mathcal{S}_{2,1}^{E}(\rho_1,\beta_2)\right) \end{equation}
is  an ECT-system on   $\hat{I}_1$, where
 \begin{equation*}
\begin{split}
&\mathcal{C}_{2,k}^{E}(\rho,\beta)=\int_{E}\cos (k\theta)(1-\cos \theta)\left(1+\rho(1-\cos\theta)\right)^{\beta}d\theta, \ \ k=0,1, \cdots, 2m-1, \\
&\mathcal{S}_{2,k}^{E}(\rho,\beta)=\int_{E}\sin (k\theta)(1-\cos \theta)\left(1+\rho(1-\cos\theta)\right)^{\beta}d\theta, \ \ k=1, 2,  \cdots, 2m-1.
\end{split}
\end{equation*}
By the transformation $\rho_2=2\rho_1+1$,  \eqref{equa60-4} is  an ECT-system on   $\hat{I}_1$ if and only if
\begin{equation} \label{equa60-5}\left(\rho_2^{\beta_2}, \mathcal{C}_{3,0}^{E}(\rho_2,\beta_2), \cdots, \mathcal{C}_{3,2m-1}^{E}(\rho_2,\beta_2), \mathcal{S}_{3,2m-1}^{E}(\rho_2,\beta_2), \cdots, \mathcal{S}_{3,1}^{E}(\rho_2,\beta_2)\right) \end{equation}
is  an ECT-system on   $\hat{I}_2$,  where $\hat{I}_2=\{ 2\rho+1:  \rho\in \hat{I}_1 \}$ and
 \begin{equation*}
\begin{split}
& \quad \quad \mathcal{C}_{3,k}^{E}(\rho,\beta)=2\mathcal{C}_{2,k}^{E}(\rho_1,\beta)\big|_{2\rho_1+1\to \rho}=\int_{E}\cos (k\theta)(1-\cos \theta)\left(1+\cos\theta+\rho(1-\cos\theta)\right)^{\beta}d\theta,\\
&  k=0,1, \cdots, 2m-1,\quad \quad\quad \quad\quad \quad\quad \quad\quad \quad\quad \quad\quad \quad\quad \quad\quad \quad\quad \quad\quad \quad\quad \quad\quad \quad\quad \quad\quad \quad\quad \quad\quad \quad \\
&\quad \quad \mathcal{S}_{3,k}^{E}(\rho,\beta)=2\mathcal{S}_{2,k}^{E}(\rho_1,\beta)\big|_{2\rho_1+1\to \rho}=\int_{E}\sin (k\theta)(1-\cos \theta)\left(1+\cos\theta+\rho(1-\cos\theta)\right)^{\beta}d\theta,\\
&   k=1, 2,  \cdots, 2m-1.
\end{split}
\end{equation*}

By multiplying   by $\rho_2^{-\beta_2}$ each function in \eqref{equa60-5}, it yields that  \eqref{equa60-5} is  an ECT-system on   $\hat{I}_2$ if and only if

\begin{equation} \label{equa60-6}\left(1, \mathcal{C}_{4,0}^{E}(\rho_3,\beta_2), \cdots, \mathcal{C}_{4,2m-1}^{E}(\rho_3,\beta_2), \mathcal{S}_{4,2m-1}^{E}(\rho_3,\beta_2), \cdots, \mathcal{S}_{4,1}^{E}(\rho_3,\beta_2)\right) \end{equation}
is  an ECT-system on   $\hat{I}_3$,  where $\hat{I}_3=\{ \rho^{-1}:  \rho\in \hat{I}_2 \}$ and
\begin{equation*}
\begin{split}
&\mathcal{C}_{4,k}^{E}(\rho,\beta)=\int_{E}\cos (k\theta)(1-\cos \theta)\left(1-\cos\theta+\rho(1+\cos\theta)\right)^{\beta}d\theta,\ \ k=0,1, \cdots, 2m-1, \\
&\mathcal{S}_{4,k}^{E}(\rho,\beta)=\int_{E}\sin (k\theta)(1-\cos \theta)\left(1-\cos\theta+\rho(1+\cos\theta)\right)^{\beta}d\theta, \ \ k=1, 2,  \cdots, 2m-1.
\end{split}
\end{equation*}

 Furthermore,   from  Lemma \ref{lem3-3},   \eqref{equa60-6} is  an ECT-system on   $\hat{I}_3$ if and only if
 \begin{equation} \label{equa60-7}\left(\mathcal{C}_{5,0}^{E}(\rho_3,\beta_3), \cdots, \mathcal{C}_{5,2m-1}^{E}(\rho_3,\beta_3), \mathcal{S}_{5,2m-1}^{E}(\rho_3,\beta_3), \cdots, \mathcal{S}_{5,1}^{E}(\rho_3,\beta_3)\right) \end{equation}
is  an ECT-system on   $\hat{I}_3$,
where
\begin{equation}\label{defCS5}
\begin{split}
&\mathcal{C}_{5,k}^{E}(\rho,\beta)=\int_{E}\cos (k\theta)\sin^2 \theta\left(1-\cos\theta+\rho(1+\cos\theta)\right)^{\beta}d\theta, \ \ k=0,1, \cdots, 2m-1, \\
&\mathcal{S}_{5,k}^{E}(\rho,\beta)=\int_{E}\sin (k\theta)\sin^2 \theta\left(1-\cos\theta+\rho(1+\cos\theta)\right)^{\beta}d\theta, \ \ k=1, 2,  \cdots, 2m-1.
\end{split}
\end{equation}

Next, we shall use the Theorem 2 of \cite{HLZ2023} to  prove that \eqref{equa60-7} is  an ECT-system on   $\hat{I}_3$. To this end, let's rewrite the functions   in \eqref{defCS5} as
\begin{equation*}\label{defCS6}
\begin{split}
&\mathcal{C}_{5,k}^{E}(\rho,\beta)=\int_{E}\cos (k\theta)\sin^2\theta (1-\cos\theta)^\beta\left(1+\rho \cot^2\frac{\theta}{2}\right)^{\beta}d\theta, \ \ k=0,1, \cdots, 2m-1, \\
&\mathcal{S}_{5,k}^{E}(\rho,\beta)=\int_{E}\sin (k\theta)\sin^2\theta (1-\cos\theta)^\beta\left(1+\rho \cot^2\frac{\theta}{2}\right)^{\beta}d\theta, \ \ k=1, 2,  \cdots, 2m-1.
\end{split}
\end{equation*}

 Recall that the Theorem 2 of \cite{HLZ2023} says that     the ordered set of functions
 \eqref{eqhlz2},
where $I_{i,j}(y)$ is defined by
 \eqref{eqhlz3}
,   form an ECT-system on $U$ if
 $G_i$'s are of the form
  \begin{align*}\label{eq3}
\begin{split}
   G=\frac{1}{(1-yg(t))^{\beta}},\ \quad
  G_i=G|_{E_i\times U},
\end{split}
  \end{align*}
  with $g$ being monotonic on $E$ and $\beta\in\big(\mathbb R\backslash\mathbb Z^-_0\big)\bigcup\big(\mathbb Z^-_0\cap(-\infty,1-{\rm Card(\mathcal F)}]\big)$ and   for each $i\in\{0,\ldots,n\}$ the ordered set of functions $\{f_{i,0},\ldots,f_{i,m_i}\}$ is a CT-system on $E_i$.

  In order to apply  the Theorem 2 of \cite{HLZ2023} for  \eqref{equa60-7}, we let
  \begin{equation}\label{usthm2ofHLZ-1}
\left(f_{0,0},\cdots,f_{0,4m-1}\right)=\left(\xi, \xi c_1,  \cdots,   \xi c_{2m-1}, \xi s_{2m-1}, \cdots, \xi s_{1}\right),
\end{equation}
where $\xi(t)=\sin^2t (1-\cos t)^{\beta_3}$, and let $G(t,y)=\left(1-yg(t)\right)^{\beta_3}$, where
$g(t)=-\cot^2\frac{t}{2}$.
 From Lemma \ref{lem3-1}, it is easy to see that   \eqref{usthm2ofHLZ-1} is an  ECT-system on  $E$.
  Since  $g(t)$ is monotonic on $(0, \pi)$, we conclude that  \eqref{equa60-7} is  an ECT-system on   $\hat{I}_3$.  This means that
 \eqref{equa60-1} is  an ECT-system on   $\hat{I}$.

\end{proof}

\section{ Proof of the main results}

The goal of this section is to prove Theorem \ref{thm1-1}.  Theorem \ref{Theorem1-3} can be regarded as a direct corollary of Theorem \ref{thm1-1}.
 For the sake of brevity, throughout this section we will denote $\mathcal{C}_{k}^{E}(\rho,\alpha),
\mathcal{S}_{k}^{E}(\rho,\alpha), \mathcal{D}_{k}^{E}(\rho,\alpha)$ respectively by $\mathcal{C}_{k}^{E},
\mathcal{S}_{k}^{E}, \mathcal{D}_{k}^{E}$.

\begin{proof}[Proof of Theorem \ref{thm1-1}]

For   clarity we divide the proof into two parts.

(i) Consider the first order Melnikov function $M_1$ obtained in Proposition \ref{prop1}.
If $\theta_1\in (0,\pi)\cup (\pi, 2\pi)$,  then from  the conclusion (i) of Proposition \ref{prop3-5},
$$\left(\mathcal{C}_{0}^{E_1}, \cdots, \mathcal{C}_{m}^{E_1}, \mathcal{S}_{m}^{E_1}, \cdots, \mathcal{S}_{1}^{E_1},\mathcal{C}_{0}^{E_2}, \cdots, \mathcal{C}_{m}^{E_2}\right)$$
with $\beta=\alpha-1$, is an ECT-system. Thus, by Lemma \ref{lem3-3},   the ordered set of functions
$$\left(1, \mathcal{C}_{0}^{E_1}, \cdots, \mathcal{C}_{m}^{E_1}, \mathcal{S}_{m}^{E_1}, \cdots, \mathcal{S}_{1}^{E_1},\mathcal{C}_{0}^{E_2}, \cdots, \mathcal{C}_{m}^{E_2}\right),$$
with $\beta=\alpha$, is an ECT-system. Note that
$$p_{10},  \mu_0^{(1)}, \cdots,  \mu_{m}^{(1)}; \mu_0^{(2)}, \cdots,\mu_{m}^{(2)}; \lambda_1, \cdots,\lambda_{m}$$
are independent, where
$  \mu_k^{(1)}=\tilde{d}_{1k}^++\tilde{d}_{1k}^-,  \mu_k^{(2)}=2\tilde{d}_{1k}^-$ for $k=0,1,\cdots, m$ and $\lambda_{k}=\tilde{c}_{1k}^+-\tilde{c}_{1k}^-$ for $k=1,\cdots, m$,
it follows  that  $M_1$ can has at most  $3m+2$ isolated zeros  on $\hat{I}$,  counting multiplicity. Moreover, this upper bound can be reached
for suitable coefficients.

Similarly, if $\theta_1=\pi$ (resp. $\theta_1=2\pi$), then  $M_1$ can has at most $2m+1$ (resp. $m+1$)  isolated zeros  on $\hat{I}$,  counting multiplicity, and this upper bound can be reached.

(ii) Next,
  assume  that  $M_1(\rho)\equiv 0$.  Let's  consider the second order   Melnikov function $M_2$ obtained in Proposition \ref{prop5},  Proposition \ref{pro5cor} and Proposition \ref{cor1}.
 In what follows we should split the discussion into two situations according to the position of $\theta_1$,.

Case 1.    $\theta_1\in (0,\pi)\cup (\pi, 2\pi)$.   By  Lemma  \ref{lem3-3} and the conclusion (iii) of Proposition \ref{prop3-5},
the ordered set of functions
\begin{equation*}  \begin{split}&\bigg(1, \mathcal{C}_{0}^{E_1}, \mathcal{C}_{1}^{E_1}, \mathcal{S}_{1}^{E_1}, \mathcal{D}_{1}^{E_1},  \cdots, \mathcal{C}_{m-1}^{E_1}, \mathcal{S}_{m-1}^{E_1}, \mathcal{D}_{m-1}^{E_1}, \mathcal{C}_{m}^{E_1}, \cdots, \mathcal{C}_{2m-1}^{E_1}, \mathcal{S}_{2m-1}^{E_1}, \cdots, \mathcal{S}_{m}^{E_1},\\
&\quad \mathcal{C}_{0}^{E_2}, \mathcal{C}_{1}^{E_2}, \mathcal{S}_{1}^{E_2}, \mathcal{D}_{1}^{E_2},\cdots, \mathcal{C}_{m-1}^{E_2}, \mathcal{S}_{m-1}^{E_2}, \mathcal{D}_{m-1}^{E_2}, \mathcal{C}_{m}^{E_2}, \cdots, \mathcal{C}_{2m-1}^{E_2}\bigg),
 \end{split}  \end{equation*}
with $\beta=\alpha$, is an ECT-system on  $\hat{I}$. It turns out that
 the maximum number of isolated zeros  of  $M_2$ counting multiplicity, is less or equals to $9m-4$, i.e.,  $Z_2(m)\leq 9m-4$.
Unfortunately, we can not prove that $Z_2(m)=9m-4$ due to  the dependence of the coefficients in   \eqref{equam2pro-3}.  Thus, we   turn  to study the lower bound for  $Z_2(m)$. Let $\beta_1=1 $ and $\beta_2=0$,  \eqref{equam2pro-3} becomes
\begin{equation*} \label{60}
M_2(\rho)=p_{20}+\sum_{k=1}^{m-1}\eta_{k}\mathcal{D}_{k}^{E_1}(\rho,\alpha)
+\sum_{k=1}^{2m-1}\lambda_k\mathcal{S}_{k}^{E_1}(\rho,\alpha)
+\sum_{k=0}^{2m-1}\left(\mu_k^{(1)}\mathcal{C}_{k}^{E_1}(\rho,\alpha)+\mu_k^{(2)}\mathcal{C}_{k}^{E_2}(\rho,\alpha)\right),
\end{equation*}
where   ${\eta}_1,\cdots, {\eta}_{m-1}$;  ${\lambda}_1, \cdots, {\lambda}_{2m-1}$;  ${\mu}_0^{(1)}, \cdots, {\mu}_{2m-1}^{(1)}$; ${\mu}_0^{(2)}, \cdots, {\mu}_{2m-1}^{(2)}$ are independent.
Since  the order set of functions \eqref{equa49-1}
is an ECT-system on $\hat{I}$, by applying the properties of Chebyshev systems,   and taking into account that $p_{20}$ is not can be proved to be  dependence of  ${\eta}_i, {\lambda}_j,  {\mu}_k^{(1)},  {\mu}_{\ell}^{(2)}$,  we conclude that there exists   a group of real numbers  ${\eta}_1,\cdots, {\eta}_{m-1}$;  ${\lambda}_1, \cdots, {\lambda}_{2m-1}$;  ${\mu}_0^{(1)}, \cdots, {\mu}_{2m-1}^{(1)}$; ${\mu}_0^{(2)}, \cdots, {\mu}_{2m-1}^{(2)}$ such that
$ {M}_2(\rho)$
 have exactly   $7m-3$  isolated   simple  zeros  in $ \hat{I}$.

 This verify that  $ 7m-3\leq Z_2(m)\leq 9m-4$.

Case 2.    $\theta_1=2\pi$.  By  Lemma  \ref{lem3-3} and the conclusion (i) of Proposition \ref{prop3-5},
the ordered set of functions
\begin{equation*} \bigg(1, \mathcal{C}_{0}^{E}, \mathcal{C}_{1}^{E},   \cdots, \mathcal{C}_{2m-1}^{E}\bigg) \end{equation*}
with $\beta=\alpha$ and $E=[0,\pi]$, is an ECT-system on  $\hat{I}$.    From Proposition  \ref{pro5cor},       it holds  that $Z_2(m)=2m-1$.

Case 3.    $\theta_1=\pi$. Under the assumption $(\mathbf{H})$, i.e.,  $a_{1k}^+=a_{1k}^-$ ($k=1, 2, \cdots, m$),  it follows from   Proposition \ref{cor1}  and Proposition  \ref{prop3-7} directly that
$M_2(\rho)$ can has at most $4m+1$ isolated zeros  on $\hat{I}$. This means that
  $Z_2(m)\geq 4m+1$.

The proof is finished.
\end{proof}

\begin{rem}  (i) In the proof of Theorem \ref{thm1-1}, when we consider $M_2(\rho)$ for $\theta_1\neq\pi$,   we are not able to provide the accurate value for $Z_2(m)$ due to the following difficulties.

 When $\theta_1\in (0,\pi)\cup (\pi, 2\pi)$,  the    difficulty  arise from   the  dependence  of coefficients of
$\mathcal{D}_{k}^{E_1}(\rho),$ $\mathcal{D}_{k}^{E_2}(\rho)$ and $\mathcal{S}_{k}^{E_2}(\rho)$ in    $M_2(\rho)$. This can be seen from the rearrangement of \eqref{equam2pro-3}
\begin{equation*} \label{equam2pro-33}
\begin{split}
M_2(\rho)&=p_{20}+\sum_{k=1}^{m-1}\eta_{k}\left(\beta_1\mathcal{D}_{k}^{E_1}(\rho,\alpha)
+\beta_2(\mathcal{D}_{k}^{E_2}(\rho,\alpha)-\pi\mathcal{S}_{k}^{E_2}(\rho,\alpha))\right)\\
&+\sum_{k=1}^{2m-1}\lambda_k\mathcal{S}_{k}^{E_1}(\rho,\alpha)
+\sum_{k=0}^{2m-1}\left(\mu_k^{(1)}\mathcal{C}_{k}^{E_1}(\rho,\alpha)+\mu_k^{(2)}\mathcal{C}_{k}^{E_2}(\rho,\alpha)\right).
\end{split}\end{equation*}
In this paper we only study the case that $ \beta_2=0$, the case  $ \beta_2\neq 0$
remain open for   further research.

(ii) When  $\theta_1=\pi$,   the    difficulty  is that,    the expreesion of   $M_2(\rho)$ contains not only the functions in \eqref{equa60-1}, but also the  functions of the forms
$$  \int_{0}^{\pi}\frac{1}{\sin^2 \theta} y_0^{\alpha}(\theta,\rho)d \theta,     \int_{0}^{\pi}\bigg(\frac{\theta}{\sin \theta}\bigg)' y_0^{\alpha}(\theta,\rho)d \theta,   \int_{0}^{\pi}\bigg(\frac{\cos(k\theta)}{\sin \theta}\bigg)' y_0^{\alpha}(\theta,\rho)d \theta, \int_{0}^{\pi}\bigg(\frac{\theta\cos(k\theta)}{\sin\theta}\bigg)' y_0^{\alpha}(\theta,\rho)d \theta.$$

In this paper  we can only  give the  accurate value for $Z_2(m)$ by imposing   the condition $(\mathbf{H})$.   Without this restriction, the expression of
$M_2(\rho)$ is very complex, and the estimation of  $Z_2(m)$  is still  a challenging problem.

\end{rem}

\section*{Acknowledgements}
The first  author is supported by    the NNSF of China (No.12571166),
  the NNSF
of Guangdong Province  (No.2023A1515010885),  and the project of promoting research capabilities for key constructed disciplines in Guangdong Province (No.2021ZDJS028).
The second  author is supported by    the NNSF of China (No.12271212).

\section*{Conflict of interest statement}
Authors state no conflict of  interest.

\end{document}